\documentclass[11pt]{amsart}
\usepackage{amssymb, amstext, amscd, amsmath, amssymb}
\usepackage{mathtools, paralist, color, dsfont, rotating}
\usepackage{verbatim}
\usepackage{enumerate,enumitem}
\usepackage{tikz-cd}
\usepackage{float}
\usepackage{setspace}
\usepackage{todonotes}

\usepackage{hyperref}
\usepackage{tikz}

\floatstyle{boxed} 
\restylefloat{figure}

\numberwithin{equation}{section}
\usepackage[a4paper]{geometry}
\usepackage{quoting}
\quotingsetup{vskip=.1in}
\quotingsetup{leftmargin=.17in}
\quotingsetup{rightmargin=.17in}
\let\OLDthebibliography\thebibliography
\renewcommand\thebibliography[1]{
  \OLDthebibliography{#1}
  \setlength{\parskip}{0pt}
  \setlength{\itemsep}{2pt plus 0.5ex}
}

\makeatletter
\def\@cite#1#2{{\m@th\upshape\bfseries%
[{#1\if@tempswa{\m@th\upshape\mdseries, #2}\fi}]}}
\makeatother
\theoremstyle{plain}
\newtheorem{theorem}{Theorem}[section]
\newtheorem{corollary}[theorem]{Corollary}
\newtheorem{proposition}[theorem]{Proposition}
\newtheorem{lemma}[theorem]{Lemma}
\theoremstyle{definition}
\newtheorem{definition}[theorem]{Definition}
\newtheorem{example}[theorem]{Example}

\newtheorem{remark}[theorem]{Remark}

\newtheorem{question}[theorem]{Question}

\theoremstyle{remark}
\newcommand{\bbC}{{\mathbb{C}}}
\newcommand{\bbD}{{\mathbb{D}}}

\newcommand{\bbF}{{\mathbb{F}}}

\newcommand{\bbN}{{\mathbb{N}}}

\newcommand{\bbT}{{\mathbb{T}}}

\newcommand{\A}{{\mathcal{A}}}
\newcommand{\B}{{\mathcal{B}}}
\newcommand{\C}{{\mathcal{C}}}
\newcommand{\D}{{\mathcal{D}}}

\newcommand{\F}{{\mathcal{F}}}

\renewcommand{\H}{{\mathcal{H}}}

\newcommand{\J}{{\mathcal{J}}}
\newcommand{\K}{{\mathcal{K}}}
\renewcommand{\L}{{\mathcal{L}}}
\newcommand{\M}{{\mathcal{M}}}
\newcommand{\N}{{\mathcal{N}}}
\renewcommand{\O}{{\mathcal{O}}}

\newcommand{\Q}{{\mathcal{Q}}}
\newcommand{\R}{{\mathcal{R}}}

\newcommand{\T}{{\mathcal{T}}}

\newcommand{\V}{{\mathcal{V}}}
\newcommand{\W}{{\mathcal{W}}}

\newcommand{\fF}{{\mathfrak{F}}}

\newcommand{\fL}{{\mathfrak{L}}}

\newcommand{\rC}{{\mathrm{C}}}

\renewcommand{\phi}{\varphi}
\newcommand{\upchi}{{\raise.35ex\hbox{\ensuremath{\chi}}}}

\newcommand{\id}{{\operatorname{id}}}

\newcommand{\spn}{\operatorname{span}}

\newcommand\Span{\mathop{\rm span}}

\newcommand{\ol}{\overline}

\newcommand{\cstarlattice}{\text{C$^*$-Lat}}

\begin{document}
%%%%%%%%%%%%%%%%%%%%%%%%%%%%%%%%%%%%%%
\title[Couniversal C*-algebras of RFD operator algebras]{Couniversality for C*-algebras of residually\\ finite-dimensional operator algebras}

\author{Adam Humeniuk}
\address{Department of Mathematics and Computing, Mount Royal University, Calgary, AB, Canada}
\email{ahumeniuk@mtroyal.ca}

\author{Christopher Ramsey}
\address{Department of Mathematics and Statistics, MacEwan University, Edmonton, AB, Canada}
\email{ramseyc5@macewan.ca}

\author{Ian Thompson}
\address{Department of Mathematical Sciences, University of Copenhagen, Universitetsparken 5, 2100 Copenhagen, Denmark}
\email{ian@math.ku.dk}


%%%%%%%%%%%%%%%%
\begin{abstract}
The $\rC^*$-envelope of a non self-adjoint operator algebra is known to encode many properties of the underlying subalgebra. However, the $\rC^*$-envelope does not always encode the residual finite-dimensionality of an operator algebra. To elucidate this failure, we study couniversal existence in the space of residually finite-dimensional (RFD) $\rC^*$-algebras attached to a fixed operator algebra. We construct several examples of residually finite-dimensional operator algebras for which there does not exist a minimal RFD $\rC^*$-algebra, answering a question of the first two authors. For large swathes of tensor algebras of $\rC^*$-correspondences, we also prove that the space of RFD $\rC^*$-algebras fails to be closed under infima of $\rC^*$-covers. In the case of the disc algebra, we are able to achieve this failure for a single pair of RFD $\rC^*$-algebras.
\end{abstract}
%%%%%%%%%%%%%%%%

\maketitle

\section{Introduction}\label{s:intro}

A central philosophy of non self-adjoint operator algebra theory has been that different (completely isometric) representations of an operator algebra can accentuate distinct properties of the algebra. As a first example, consider the disc algebra $\mathrm{A}(\bbD)\subset\rC(\ol{\bbD})$, which consists of those continuous functions that restrict to be holomorphic on the open unit disc. The canonical generator of the disc algebra is a universal contraction on account of the von Neumann inequality \cite[Corollary 3.12]{Paulsen}. Thus, $\mathrm{A}(\bbD)$ admits a (generator preserving) representation into the universal $\rC^*$-algebra generated by a contraction. On the other hand, the disc algebra also admits representations into the universal $\rC^*$-algebra generated by a unitary, or an isometry. Each of these representations exemplifies a $\rC^*$-algebra with unique qualities: one admits many finite-dimensional representations while not being completely captured by them \cite[Proposition 3.3 and Example 3.4]{CourtSherm}, another is completely understood by its one-dimensional representations, whereas the final $\rC^*$-algebra is not even stably finite.

In hopes of capturing the $\rC^*$-algebraic information intrinsic to an operator algebra, Arveson proposed a non-commutative analogue to the Shilov boundary of a uniform algebra \cite{Arv69}. Categorically, this corresponds to the couniversal (or minimal) $\rC^*$-algebra generated by a completely isometric copy of the given algebra. The existence of such an object, now referred to as the $\rC^*$-envelope of the operator algebra, was first established by Hamana \cite{Hamana}. Since then, the $\rC^*$-envelope has dominated the modern understanding of non self-adjoint operator algebra theory \cite{Arv08},\cite{CH},\cite{CT},\cite{DKChoq},\cite{DKKLL},\cite{HumRam}. This is helped, in large part, because the $\rC^*$-envelope serves as a connective link between self-adjoint and non self-adjoint operator algebra theory; with many properties of an operator algebra being able to be encoded in those of its $\rC^*$-envelope \cite{DKKLL},\cite{DS},\cite{DT},\cite{HumRam},\cite{KatRamMem}. Despite this, the $\rC^*$-envelope can also be insufficient for capturing many properties of an operator algebra.

This is especially true when considering the residual finite-dimensionality of an operator algebra \cite{CR}. An operator algebra $\A$ is residually finite-dimensional (RFD) if, in an appropriate sense, $\A$ can be completely recovered from its finite-dimensional representations (see Definition \ref{d:rfd}). In the non self-adjoint setting, residual finite-dimensionality has become a recent trend, with many questions and mysteries that surround the myriad of different characterizations \cite{CDO},\cite{CM},\cite{CR},\cite{Hartz},\cite{Kat},\cite{MP},\cite{Thompson}. For one, residual finite-dimensionality of non self-adjoint algebras allows for a much more flexible class than that in the $\rC^*$-algebraic canon, such as the unital subalgebra $\A_d\subset\O_d$ generated by the Cuntz isometries. On the other hand, the $\rC^*$-envelope can grossly fail to encode the residual finite-dimensionality of the subalgebra \cite[Section 6]{CR},\cite[Section 4]{CT}. Indeed, simplicity of the Cuntz algebra guarantees that $\O_d$ is the $\rC^*$-envelope of $\A_d$, yet $\O_d$ admits no nonzero finite-dimensional $*$-representations whatsoever. Accordingly, one might hope to uncover an appropriate substitute for the $\rC^*$-envelope, that captures the residual finite-dimensionality of the operator algebra. The question of whether such a substitute can exist was previously raised by the first two authors in \cite[Section 5.1]{HumRam}.

In this work, we demonstrate that the existence of such an object is impossible; both for many well-studied examples and in many different formulations of the problem at hand. To this end, we fix an operator algebra $\A$ that is residually finite-dimensional. The collection of all $\rC^*$-algebras generated by a copy of $\A$, which we call the C*-covers of $\A$, can then be given the structure of a complete lattice, with the supremum and infimum (respectively, join and meet) operations corresponding to universal and couniversal $\rC^*$-algebras associated with a given family \cite{Hamidi}. Our central question now reduces to understanding closure properties for the space of those residually finite-dimensional $\rC^*$-covers for $\A$ under the given infimum operation. %We consider the subcollection of those $\rC^*$-covers of $\A$ determining a residually finite-dimensional $\rC^*$-algebra.

In Section \ref{s:prelim}, we record terminology and the inner-workings of this lattice structure. Section \ref{s:rfd-semilattice} includes a brief introduction to residual finite-dimensionality in the non self-adjoint context. Following this, most of Section \ref{s:rfd-semilattice} is devoted to proving that the space of RFD $\rC^*$-algebras of large classes of tensor algebras fail to be closed under infimums (Theorem \ref{t:tens-lat}). In Section \ref{s:disc-alg}, a separate analysis allows us to extrapolate a stronger conclusion for the disc algebra (Theorem \ref{t:disc-lattice}), by explicitly constructing exactly two RFD C*-algebras whose infimum is not RFD. Finally, in Section \ref{s:rfd-env}, we provide several examples of an RFD operator algebra $\A$ where there exists no minimal RFD $\rC^*$-algebra generated by $\A$. In other words, there is no appropriate substitute for the $\rC^*$-envelope that also encodes the residual finite-dimensionality of the operator algebra. One such example of this failure is for the subalgebra $\A_d\subset\O_d$, by means of an appropriate Wold-type decomposition for row contractions (Theorem \ref{t:pop-rfd-min}).

As the arguments needed to establish our conclusions are somewhat technical, here we present the common thread underlying our constructions. First, we consider a natural choice of a block-diagonal representation of the operator algebra with finite-dimensional blocks. We then form a downward directed chain of dilations by successively removing finite-dimensional blocks from the representation. Then, by employing a variety of techniques, we concretely identify the representation at the limit; directly proving that the limiting $\rC^*$-algebra is not residually finite-dimensional. The fundamental idea of considering an inductive limit of a chain of dilations appears to be reminiscent of the Dritschel--McCullough dilation theorem \cite{DM}, which is one of the most readily--accessible ways of constructing the $\rC^*$-envelope of an operator algebra. However, the Dritschel--McCullough dilation theorem is an upward directed chain of dilations, whereas our analysis is moreso focused on constructing a downward directed chain.

%%%%%%%%%%%%%%%%%%%%%%%%%%%%%%%%%%%%%%%%%%
%%%%%%%%%%%%%%%%%%%%%%%%%%%%%%%%%%%%%%%%%%

%Suppose $\A$ is an operator algebra. It is currently unknown whether the subset of RFD C$^*$-covers of the complete lattice of C$^*$-covers, $\cstarlattice(\A)$, is a complete lattice, a lattice, or has a minimal element. This problem was posed in \cite{HumRam}. The set is a complete upper semilattice since the arbitrary direct sum of RFD C$^*$-covers is RFD. Thompson \cite{Thompson} pointed this out and showed that this implies there is a maximal RFD C$^*$-cover. Hartz \cite{Hartz} gave an example to show that this maximal RFD C$^*$-cover does not need to be the maximal C$^*$-cover.

%Sometimes the RFD C$^*$-covers do form a complete lattice. For instance, in \cite{HumRam} it is proven that every C$^*$-cover of $\T_2$, the upper triangular $2\times 2$ matrices, is RFD. There are also lots of examples where there is a minimal RFD C$^*$-cover, specifically anytime the C$^*$-envelope is RFD. However, there are many examples where an RFD operator algebra $\A$ has a C$^*$-envelope that is not RFD. For instance, the non-commutative disc algebra $\A_2$ is known to be RFD \cite{CR} but its C$^*$-envelope is $\O_2$ which is simple.

\section{Preliminaries}\label{s:prelim}

An {\bf operator algebra} is a subalgebra $\A\subset B(\H)$ of Hilbert space operators. Equivalently, $\A$ is an algebra for which there is a sequence of norms on the matrix algebras $M_n(\A)$ that satisfy appropriate compatibility conditions \cite[Section 16]{Paulsen}. A {\bf representation} of an operator algebra $\A$ is a completely contractive algebra homomorphism $\rho:\A\rightarrow B(\H)$. Throughout, $M_n$ will denote the algebra of complex-valued $n\times n$ matrices and $K(\H)$ will denote the algebra of compact operators on $\H$. In this section, we record terminology and machinery on the C*-algebras generated by a completely isometric copy of $\A$. For further information on non self-adjoint operator algebras and the details that we outline in this section, the reader may consult \cite{BlechLM},\cite[Section 3]{Hamidi},\cite[Section 2]{HumRam}, or \cite{Paulsen}.

\begin{definition}
    A {\bf C*-cover} of $\A$ is a pair $(\B, \iota)$ consisting of a $\rC^*$-algebra $\B$ and a completely isometric homomorphism $\iota:\A\rightarrow \B$ such that $\rC^*(\iota(\A)) = \B$.
\end{definition}

Given a pair of $\rC^*$-covers $(\B, \iota), (\C, j)$ for $\A$, a {\bf morphism} $\pi:(\C, j)\rightarrow (\B, \iota)$ of $\rC^*$-covers is a $\ast$-homomorphism $\pi:\C\rightarrow \B$ such that $\pi j = \iota$. Since $\rC^*(j(\A)) = \C$, there is at most one morphism $(\C, j)\rightarrow (\B, \iota)$. Furthermore, as in \cite{Arv08}, the closed two-sided ideal $\ker\pi$ associated to a morphism $\pi:(\C, j)\rightarrow (\B, \iota)$ of $\rC^*$-covers is called a {\bf boundary ideal}.

The central results of this manuscript will be expressed in terms of an appropriately selected order structure on the $\rC^*$-covers of $\A$. Specifically, we define a preorder by $(\B, \iota)\preceq (\C, j)$ if and only if there is a morphism $\pi:(\C, j)\rightarrow (\B, \iota)$. The $\rC^*$-covers $(\B, \iota)$ and $(\C, j)$ are said to be {\bf equivalent} if there are morphisms $(\B, \iota) \rightarrow (\C, j)$ and $(\C, j) \rightarrow (\B, \iota)$. Equivalently, two $\rC^*$-covers are equivalent if there is a $*$-isomorphism $\pi:\C\rightarrow \B$ such that $\pi j = \iota$. Throughout this manuscript, the equivalence class of a $\rC^*$-cover $(\B, \iota)$ will be denoted by $[\B, \iota]$. The preorder $\preceq$ descends to a partial order on the set of equivalence classes, which we denote by the same symbol. Furthermore, whenever appropriate, we will freely identify a $\rC^*$-cover with its equivalence class.

The ordering of the $\rC^*$-covers of an operator algebra gives rise to a natural identification of the two most frequently studied $\rC^*$-covers: the {\bf maximal $\rC^*$-cover}, denoted $[\rC^*_{\max}(\A), \mu]$, and the {\bf minimal $\rC^*$-cover}, denoted $[\rC^*_e(\A), \varepsilon]$. These $\rC^*$-covers correspond to the (unique) maximal and minimal elements of the ordering, and the existence of these $\rC^*$-covers is always guaranteed on account of \cite{Blech},\cite{Hamana}. The minimal $\rC^*$-cover (or {\bf $\rC^*$-envelope} of $\A$) has been the subject of particularly intense focus (see \cite{Arv08},\cite{CH},\cite{DKChoq},\cite{DT},\cite{KatRamMem} and the references therein), dating back to Arveson's seminal work on subalgebras of $\rC^*$-algebras \cite{Arv69}. Our manuscript expands upon this trend within the framework of residual finite-dimensionality for operator algebras.

As recognized in \cite[Section 2.8]{BlechLM},\cite{Hamidi}, the partial ordering $\preceq$ on $\rC^*$-covers promotes to a complete lattice structure on the set of equivalence classes. Accordingly, the set of equivalence classes of $\rC^*$-covers for $\A$ will be denoted by $\cstarlattice(\A)$, and referred to as the {\bf $\rC^*$-lattice} of $\A$. The lattice operations of $\cstarlattice(\A)$ correspond to the following categorical constructs: the join in $\cstarlattice(\A)$ corresponds to the smallest $\rC^*$-cover that is universal for the family $\F$ , whereas the meet is the largest $\rC^*$-cover that is couniversal for $\F$.

The lattice structure of $\cstarlattice(\A)$ will be central to our arguments, and so we explicitly describe the meet and join operations. Let $\F = \{ [\B_i, \iota_i]: i\in I\}$ be a collection of $\rC^*$-covers in $\cstarlattice(\A)$. A representative for the join of $\F$ can be selected as \[\bigvee_{i\in I}[\B_i, \iota_i] = \left[ \rC^*(\iota(\A)), \iota:= \bigoplus_{i\in I}\iota_i \right],\]whereas a representative for the meet of $\F$ can be constructed by the following procedure. For each $i\in I$, there is a morphism $\pi_i:(\rC^*_{\max}(\A), \mu)\rightarrow(\B_i, \iota_i)$ and so \[\J = \ol{\sum_{i\in I}\ker\pi_i}\]defines a closed two-sided ideal of $\rC^*_{\max}(\A)$. The meet of $\F$ can then be described by \[\bigwedge_{i\in I}[\B_i, \iota_i] = [\rC^*_{max}(\A)/\J, q \mu ]\]where $q:\rC^*_{\max}(\A)\rightarrow\rC^*_{\max}(\A)/\J$ is the quotient mapping. Note that $q\mu$ is completely isometric as $\pi_i\mu$ is completely isometric for each $i\in I$.

\begin{remark}\label{r:meet}
    To describe the meet of $\F$, one may also replace $[\rC^*_{\max}(\A),\mu]$ by the join of $\F$, or any C*-cover that dominates all elements of $\F$. Indeed, since $\pi_i:\rC^*_{\max}(\A)\rightarrow\B_i$ is a surjective $\ast$-homomorphism for each $i\in I$, it follows that $\pi_i$ factors through to a $\ast$-homomorphism of $\rC^*(\iota(\A))$ where $\iota = \bigoplus_{i\in I}\iota_i$. As in \cite[Discussion following Proposition 2.4]{HumRam}, one can then identify the meet of $\F$ with an appropriate quotient of the $\rC^*$-algebra $\rC^*(\iota(\A))$.
\end{remark}

Assume that the index set $I$ is downward directed, in the sense that $[\B_i, \iota_i]\preceq[\B_j, \iota_j]$ if and only if $i\geq j$. In this setting, the meet of $\F$ admits a more convenient identification. This follows from observing that the morphisms \[\pi_{i,j}:(\B_j, \iota_j)\rightarrow (\B_i, \iota_i), \quad i\geq j,\]endow $\F$ with the structure of a directed system. The inductive limit \[\B = \varinjlim \B_i\]of the $\rC^*$-algebras along the $\ast$-homomorphisms $\{\pi_{i,j}:\B_j\rightarrow \B_i ~:~ i\geq j\}$ can then be given the structure of a $\rC^*$-cover for $\A$, which coincides with the meet of $\F$ \cite[Proposition 2.5]{HumRam}.

\begin{proposition}\label{p:meet}
    Let $\A$ be an operator algebra and $\F = \{[\B_i, \iota_i]:i\in I\}$ be a downward directed set in $\cstarlattice(\A)$. For each $i\in I$, let $\pi_i:\B_i\rightarrow \varinjlim \B_i$ be the canonical $\ast$-homomorphism arising from the directed system $\{\pi_{i,j}:(\B_j, \iota_j)\rightarrow (\B_i, \iota_i) : i\geq j\}$. Then \[\bigwedge_{i\in I}[\B_i, \iota_i]  = \left[\varinjlim \B_i, \iota \right]\]where $\iota:= \pi_i\iota_i$ for any $i\in I$.
\end{proposition}

We emphasize here that $\iota$ is independent of the choice of $i\in I$. Furthermore, this identification for the meet will be the central computational tool that we require in this manuscript.

\section{The semilattice of residually finite-dimensional C*-covers}\label{s:rfd-semilattice}

In this section, we introduce the order-theoretic conditions present for the $\rC^*$-covers in $\cstarlattice(\A)$ that are completely captured by their finite-dimensional representations. To this end, the following definition is central to the remainder of the manuscript.

\begin{definition}\label{d:rfd}
    An operator algebra $\A$ is {\bf residually finite-dimensional}, or RFD, if there is a collection $\{r_\lambda: \lambda\in\Lambda\}$ of positive integers together with a completely isometric representation $\iota: \A\rightarrow\prod_\lambda M_{r_\lambda}$.
\end{definition}

It is clear that an operator algebra $\A$ is residually finite-dimensional if and only if there is a $\rC^*$-cover $[\R, \upsilon]$ of $\A$ such that $\R$ is an RFD $\rC^*$-algebra, which will be of the form $\rC^*(\iota(\A))$ for some representation $\iota$ as above. For simplicity, we refer to $[\R, \upsilon]$ as a residually finite-dimensional $\rC^*$-cover. Large families of operator algebras are known to be residually finite-dimensional, as can be seen in \cite{CDO},\cite{CM},\cite{CR},\cite{CT},\cite{Kat},\cite{MP}. For example, the algebra of upper triangular operators in $B(\ell^2(\bbN))$, the multiplier algebra of a reproducing kernel Hilbert space, and the operator algebra generated by the Cuntz isometries are all residually finite-dimensional. %Alternatively, an operator algebra can be seen to be residually finite-dimensional precisely when \[ \|A\| = \sup\{ \| \pi^{(r)}(A)\|\}, \quad A\in M_r(\A), \ r\in\bbN,\]with the supremum being taken over all representations $\pi:\A\rightarrow B(\H)$ on a finite-dimensional Hilbert space $\H$. That is, in an appropriate sense, $\A$ may be completely understood by its finite-dimensional representations. %We record a brief list of examples here.

%\begin{example}
%    The following operator algebras are known to be residually finite-dimensional:\begin{enumerate}
%        \item Any subalgebra of the upper triangular operators in $B(\ell^2)$ \cite[Example 3.4]{CM}.
%        \item The multiplier algebra of a reproducing kernel Hilbert space \cite[Section 3]{MP}.
%        \item Tensor algebras of certain C*-correspondences, including all graph correspondences \cite[Section 4.1]{CR}. These algebras will be the primary examples that we consider in Section \ref{ss:tensor-alg}.
        %\item The full $\rC^*$-algebra $\rC^*(\bbF_2)$ of the free group $\bbF_2$ with $2$ generators \cite[Theorem 7]{Choi}.
%        \item Let $P$ be a discrete cancellative semigroup. If $P$ satisfies the so-called finite divisor property, then the reduced semigroup algebra $\A_r(P)\subset B(\ell^2(P))$ is residually finite-dimensional \cite[Proposition 4.1]{CDO}.
%    \end{enumerate}
%\end{example}

For our purposes, the collection of all RFD $\rC^*$-covers in $\cstarlattice(\A)$ will be our primary object of study. With regard to the order structure on this space, there is one known constraint that we record here.

\begin{proposition}\label{p:rfd-join}
    Let $\A$ be an operator algebra. Then the residually finite-dimensional $\rC^*$-covers of $\A$ form a (possibly empty) complete join-semilattice in $\cstarlattice(\A)$.
\end{proposition}

\begin{proof}
    This follows after noting that a direct sum of representations into RFD $\rC^*$-algebras has an image that is an RFD $\rC^*$-algebra (see \cite[Lemma 4.1]{Thompson}).
\end{proof}

Apart from the above conclusion, the space of residually finite-dimensional $\rC^*$-covers remains somewhat mysterious. There are many examples of $\rC^*$-covers\[ [\R, \iota] \preceq [\B, j] \preceq [\Q, \eta]\]of an operator algebra where $\R$ and $\Q$ are RFD, yet $\B$ is not. For an example, one may consider the disc algebra $\mathrm{A}(\bbD)\subset\rC(\ol{\bbD})$, which is the subalgebra of functions that are holomorphic on the interior. It is known that there is a diagram of $\rC^*$-covers for this operator algebra 
\[\begin{tikzcd}
      & {[\rC^*(x : \| x\|\leq 1), \mu]} \arrow[ld] \arrow[rd] & \\
       {[\mathrm{C}(\overline{\bbD}),\id]} \arrow[rd] & & {[\rC^*(S),\kappa]} \arrow[ld]\\
       & {[\mathrm{C}(\bbT),\varepsilon]} & \\
\end{tikzcd}\]
where $S\in B(\ell^2(\bbN))$ is the unilateral shift and $\rC^*(x : \|x\|\leq1)$ is the universal $\rC^*$-algebra generated by a contraction (see \cite[Example 1]{Thompson} for more details). In this case, both $\mathrm{C}(\bbT)$ and $\rC^*(x:\|x\|\leq1)$ are residually finite-dimensional \cite[Proposition 2.2 (3)]{CourtSherm}, whereas $\rC^*(S)$ is not because it contains the ideal of compact operators $K(\ell^2(\bbN))$.

On the other hand, for the algebra of upper triangular $2\times 2$ matrices $T_2\subset M_2$, it is known that every $\rC^*$-cover is residually finite-dimensional \cite[Proposition 5.1]{HumRam}. Thus, the RFD $\rC^*$-covers of $T_2$ will trivially form a complete lattice. The purpose of the remainder of this section is to establish that large classes of well-studied operator algebras do not share this conclusion.

\subsection{Tensor algebras of finite C*-correspondences}\label{ss:tensor-alg}

Here, we prove that the complete join-semilattice of RFD $\rC^*$-covers is not always a complete lattice. This is accomplished by establishing such a failure for operator algebras formed through $\rC^*$-correspondences. For additional background on $\rC^*$-algebras associated with $\rC^*$-correspondences, the authors recommend that the reader consult \cite[Section 4.6]{BO},\cite{Katsura}.

\begin{definition}
A {\bf C$^*$-correspondence} $(X,\C)$ is a C$^*$-algebra $\C$ and a $\C$-bimodule $X$ that has a $\C$-valued inner product $\langle \cdot,\cdot\rangle$ that satisfies for all $x,y\in X$, $c\in \C$
\begin{itemize}
    \item $\langle x,yc\rangle = \langle x,y\rangle c$
    \item $\langle x,y\rangle^* = \langle y,x\rangle$
    \item $\langle x,x\rangle \geq 0$ and $\langle x,x\rangle = 0$ if and only if $x=0$
    \item $\langle x,\varphi_X(c)y\rangle = \langle \varphi_X(c^*)x, y\rangle$
\end{itemize}
where the left action is $\ast$-homomorphism $\varphi_X:\C \rightarrow \L(X)$ into the C$^*$-algebra of adjointable operators on $X$.
\end{definition}

%Every C$^*$-correspondence can be realized concretely as $\C,X\subseteq B(\H)$ such that $\langle x,y\rangle = x^*y$.

\begin{definition}
Given a C$^*$-correspondence, one can define the tensor product C$^*$-correspondence $(X\otimes X, \C)$ where
\begin{itemize}
    \item $x_1c \otimes x_2 = x_1 \otimes cx_2$
    \item $\varphi_{X\otimes X}(c)(x_1\otimes x_2)d = \varphi_X(c)x_1 \otimes x_2d$
    \item $\langle x_1\otimes x_2, y_1\otimes y_2\rangle = \langle x_2, \varphi_X(\langle x_1,y_1\rangle) y_2\rangle$.
\end{itemize}
\end{definition}

The {\bf Fock space} of $X$ is the C$^*$-correspondence $(\F_X,\C)$ where
\[
\F_X = \bigoplus_{n\geq 0} X^{\otimes n}\, .
\]
and $X^{\otimes 0} = \C$. The Fock space of a $\rC^*$-correspondence gives rise to a couple frequently studied operator algebras. For this, consider the non-degenerate $\ast$-homomorphism $\rho_\infty : \C \rightarrow \L(\F_X)$ where $\rho_\infty(c)$ is given by 
\[
\rho_\infty(c) = \varphi_{\F_X}(c) = \left[\begin{matrix} c \\ & \varphi_X(c) \\ && \varphi_{X^{\otimes 2}}(c) \\ &&& \ddots \end{matrix}\right].
\]
As well, consider the linear map $t_\infty : X \rightarrow \L(\F_X)$ given by 
\[
t_\infty(x) = \left[\begin{matrix} 0 \\ t^{(0)}(x) & 0 \\ &t^{(1)}(x) & 0 \\ &&\ddots\ddots\end{matrix}\right]
\]
where $t^{(n)}(x) : X^{\otimes n} \rightarrow X^{\otimes n+1}$ is the left creation map
\[
t^{(n)}(x)(x_1\otimes \cdots \otimes x_n) = x\otimes x_1\otimes \cdots \otimes x_n
\]
and $t^{(0)}(x)(c) = xc$.

\begin{definition} 
The {\bf Toeplitz algebra} of a C$^*$-correspondence $(X,\C)$ is $\T_{(X,\C)} = \rC^*(t_\infty(X),\rho_\infty(\C))$ and the {\bf tensor algebra} is $\T_{(X,\C)}^+ = \overline{Alg\{t_\infty(X),\rho_\infty(\C)\}}$.
\end{definition}

In recent years, tensor algebras have been a primary object of study for non self-adjoint operator algebra theory (see \cite{DKDoc},\cite{DKKLL},\cite{DT},\cite{Kakariadis},\cite{KK},\cite{KatRamMem} and the references therein). This is in part because tensor algebras have demonstrated meaningful interactions with Cuntz-Pimsner $\rC^*$-algebras, as well as comprising a large class of well-studied examples \cite[Example 2.12]{HKR}. In addition, both tensor algebras and Toeplitz algebras enjoy universal properties in terms of certain classes of representations of $(X,\C)$.

\begin{definition}
A {\bf completely contractive representation} $(\rho, t, \H)$ of $(X,\C)$ consists of a $\ast$-homomorphism $\rho: \C \rightarrow B(\H)$ and a completely contractive linear map $t:X\rightarrow B(\H)$ satisfying
\[
t(xc) = t(x)\rho(c) \quad \textrm{and} \quad t(\varphi(c)x) = \rho(c)t(x)
\]
for all $x\in X, c\in\C$. 

A completely contractive representation will be called {\bf isometric} if it additionally satisfies
\[
t(x)^*t(y) = \rho(\langle x,y\rangle)
\]
for all $x,y\in X$.
\end{definition}

Muhly-Solel \cite{MS1} proved that every completely contractive representation $(\rho,t,\H)$ extends to a completely contractive representation $\rho\rtimes t : \T_X^+ \rightarrow B(\H)$, called the integrated form. Furthermore, every completely contractive representation of $\T_X^+$ is the integrated form of some representation $(\rho, t, \H)$. Using this universal property, it has been verified that some of the most frequently studied tensor algebras are residually finite-dimensional \cite[Section 4.1]{CR}. Despite this, the residually finite-dimensional representations acquired in \cite{CR} are typically distinct from the representation arising from $(\rho_\infty, t_\infty)$.

Indeed, the Toeplitz representation $(\rho_\infty, t_\infty)$ is known to be universal for isometric representations of $(X, \C)$, which also coincide with the so-called semi-Dirichlet representations of the tensor algebra $\T_X^+$ \cite{HKR},\cite{MS1}. Here, a representation $\rho:\A\rightarrow B(\H)$ of an operator algebra $\A$ is said to be {\bf semi-Dirichlet} if \[\rho(\A)^*\rho(\A)\subset\ol{\rho(\A)+\rho(\A)^*}.\]The completely isometric semi-Dirichlet representations form a downward-closed complete sublattice \cite[Theorem 3.6]{HKR}, while the residually finite-dimensional representations form a complete join-semilattice (Proposition \ref{p:meet}). Nevertheless, even for standard examples, it can be somewhat uncommon for a representation of a tensor algebra to be both semi-Dirichlet and RFD.

\begin{example}\label{e:popescu}
Let $d\geq2$ be a positive integer. Let $V_1, \ldots, V_d$ be isometries on a Hilbert space $\H$ with pairwise orthogonal ranges and such that \[\sum_{k=1}^d V_k V_k^* \neq I.\]Let $\A_d$ denote the unital norm-closed algebra generated by $\{V_1, \ldots, V_d\}$. It is known that $\A_d$ is the tensor algebra of the $\rC^*$-correspondence $(\bbC^d, \bbC)$ \cite[Example 2.7]{MS1}. Furthermore, the $\rC^*$-algebra $\T_d$ generated by $\A_d$ is the Cuntz-Toeplitz $\rC^*$-algebra, which corresponds to the Toeplitz algebra of $(\bbC^d, \bbC)$.

We claim that $\A_d$ does not have a semi-Dirichlet representation that is RFD. Indeed, by \cite[Theorem 3.10]{HKR}, the maximal semi-Dirichlet representation is the Toeplitz algebra $\T_d$. Then, on account of \cite[Theorem 3.6]{HKR}, the semi-Dirichlet representations of $\A_d$ correspond to $\rC^*$-covers of $\A_d$ that are dominated by $[\T_d, \id]$. None of the corresponding $\rC^*$-algebras can be RFD as they are quotients of $\T_d$. Indeed, $\T_d$ is not RFD as it contains the ideal of compact operators $K(\H)$ on an infinite-dimensional Hilbert space $\H$, while $\T_d/K(\H)\simeq\O_d$ is simple and infinite-dimensional. On the other hand, $\A_d$ has RFD $\rC^*$-covers by \cite[Theorem 4.6]{CR}.
\end{example}

Not withstanding the above example, there are families of tensor algebras that admit semi-Dirichlet RFD representations. If one considers the algebra $T_n\subset M_n$ of upper triangular matrices, then it can be directly verified that $T_n\simeq\T_{(\bbC^{n-1}, \bbC^n)}^+$ where $\bbC^n$ corresponds to diagonal matrices and $\bbC^{n-1}$ corresponds to the superdiagonal. In this case, the identity representation is both RFD and semi-Dirichlet. More generally, one could also consider the graph correspondence $X_G$ of a countable directed graph $G$ \cite[Example 4.6.13]{BO}. It is known that the $\rC^*$-envelope of $\T_{X_G}^+$ is the graph $\rC^*$-algebra \cite{KK}, and that the representation of $\T_{X_G}^+$ inside the graph $\rC^*$-algebra $\rC^*(G)$ is semi-Dirichlet \cite{DKDoc}. If $G$ is finite and no cycle in $G$ has an entry, then this semi-Dirichlet $\rC^*$-cover will always be RFD \cite{BelShul}.

We now work towards proving our main result on the space of residually finite-dimensional $\rC^*$-covers for tensor algebras. We start with a tool to generate $\rC^*$-covers for tensor algebras.

\begin{proposition}\label{pr:tensor-shift}
If $(\rho,t,\H)$ is a completely contractive representation of $(X,\C)$, then so is $(\rho\otimes I_\K, t\otimes V, \H\otimes \K)$ where $V\in B(\K), \|V\|\leq 1,$ and
\[
(\rho\otimes I_\K)(c) = \rho(c) \otimes I_\K \quad \textrm{and} \quad (t\otimes V)(x) = t(x) \otimes V\,.
\]
Furthermore, if $\rho\rtimes t$ is a completely isometric representation and $V=U$, the bilateral shift, or $V=S$, the unilateral shift, then $(\rho\otimes I)\rtimes (t\otimes V)$ is a completely isometric representation.
\end{proposition}
\begin{proof}
It is immediate that $(\rho\otimes I_\K, t\otimes V, \H\otimes \K)$ is a completely contractive representation of $(X,\C)$.

In \cite[Lemma 7.4]{KatRamMem} it is proven that for any $a_0,\dots, a_n\in B(\H)$ 
\[
\left\| \sum_{k=0}^n a_k\right\| \ \leq \ \left\| \sum_{k=0}^n a_k \otimes S^k \right\|\,.
\]
Using this and its matricial version, we obtain that $(\rho\otimes I)\rtimes (t\otimes S)$ is completely isometric. 

Lastly, when $V=U$ is the bilateral shift, compression to the subspace $\K_2 = \H\otimes \ell^2(\mathbb N) \subset \H\otimes\ell^2(\mathbb Z) = \K_1$ maps $(\rho\otimes I_{\K_1}, t\otimes U, \H\otimes \K_1)$ to $(\rho\otimes I_{\K_2}, t\otimes S, \H\otimes \K_2)$.
\end{proof}

Proposition \ref{pr:tensor-shift} is very similar to the so-called Extension Theorem from \cite[Theorem 7.5]{KatRamMem}. The discrepancy of Proposition \ref{pr:tensor-shift} is that the Extension Theorem \cite{KatRamMem} is concerned with isometric representations of $(X,\C)$.

Let $\{e_n: n\in\bbN\}$ denote the canonical orthonormal basis of $\ell^2(\bbN)$. Given an integer $n\ge 1$, we set $\H_n = \text{span}\{e_1,\ldots, e_n\}$ and consider the truncation $S_n:= P_{\H_n} S|_{\H_n}$ of the unilateral shift. For our purposes, the first observation that we require is the following.

\begin{corollary}
If $\rho \rtimes t$ is a completely isometric representation of $\T_X^+$ and $Y \subset \mathbb N$ is infinite, then
\[
\left(\rho \otimes \bigoplus_{n\in Y} I_n\right) \rtimes \left(t \otimes \bigoplus_{n\in Y} S_n\right)
\ = \ \bigoplus_{n\in Y} (\rho \otimes I_n) \rtimes (t \otimes S_n)
\]
is completely isometric.
\end{corollary}
\begin{proof}
The compressions to $\H\otimes \mathbb C^n$ completely norm.
\end{proof}

%Similarly, compression of the bilateral shift representation leads to the following:

%\begin{corollary}
%If $\rho \rtimes t$ is a completely isometric representation of $\T_X^+$ then so is 
%\[
%(\rho\otimes I) \rtimes (t\otimes S^*)\,.
%\]
%\end{corollary}

%Of course, you don't need this to conclude that $(\rho \otimes (I\oplus I)) \rtimes (t \otimes (S\oplus S^*))$ is a completely isometric representation.

Note that the (classical) Toeplitz algebra is not RFD because it contains the ideal of compact operators on infinite-dimensional Hilbert space. Despite this, the procedure of Proposition \ref{pr:tensor-shift} can induce $\rC^*$-covers for the tensor algebra that remain RFD.

\begin{example}
For each $n\geq 1$, notice that $(S_n \otimes S)^n = 0$. Moreover, $S_n \otimes S$ is unitarily equivalent, by a permutation, to 
\[
0 \oplus S_1 \oplus \dots \oplus S_{n-1} \oplus S_n\otimes I_{\ell_2(\H)}\,.
\]
Thus, if $\rho(1) = I$ and $t(1) = \bigoplus_{n\geq 1} S_n$ then 
\[
\rC^*(\rho\rtimes t(A(\mathbb D))) \simeq \rC^*((\rho\otimes I)\rtimes (t\otimes S)(A(\mathbb D)))
\]
and so the latter C*-algebra is still RFD.
In the same we get that
\[
\rC^*\left(\left(\rho\otimes \bigoplus_{m\geq k} I_m\right)\rtimes \left(t\otimes \bigoplus_{m\geq k} S_m\right)(A(\mathbb D))\right)
\simeq \rC^*(\rho\rtimes t(A(\mathbb D)))
\]
because 
\[
(S_n \otimes S_m)^{\inf\{n,m\}} = 0\,.
\]
\end{example}

\bigskip

A C$^*$-correspondence $(X,\C)$ will be called {\bf finite} when both $X$ and $\C$ are finite-dimensional. The tensor algebra of a finite C*-correspondence is RFD, but there are also many more cases that lead to being RFD as well \cite[Section 4.1]{CR}. 

\begin{proposition}\label{prop:finite}
If $(X,\C)$ is a finite C*-correspondence, then
\[
\rC^*((\rho_\infty \otimes I_n) \rtimes (t_\infty  \otimes S_n)(\T_X^+))
\]
is a finite-dimensional $\rC^*$-algebra for each $n\in \mathbb N$.
\end{proposition}
\begin{proof}
For any $x,y\in X$, we have 
\[
(t_\infty(x)\otimes S_n)^*(t_\infty(y)\otimes S_n) = \rho_\infty(\langle x,y\rangle) \otimes  (I_n - E_{nn})\,.
\]
As well, for any $x_1,\dots, x_n\in X$
\[
(t_\infty(x_1)\otimes S_n)\dots (t_\infty(x_n)\otimes S_n) = 0\,.
\]
Therefore, $\rC^*((\rho_\infty \otimes I_n) \rtimes (t_\infty \otimes S_n))$ is contained in $n\times n$ matrices with entries in
\[
\C + \Span\Big\{ t_\infty(x_1)\dots t_\infty(x_k)t_\infty(y_1)^*\dots t_\infty(y_m)^* : x_i,y_j\in X, 0\leq k,m < n \Big\}\,.\qedhere
\]
\end{proof}

%Certainly there are other examples of C*-correspondences where these C*-algebras are RFD.

Finally, we will require some finer facts on the representation theory of tensor algebras.

\begin{lemma}
Let $\K$ be the subspace of $\H\otimes \ell_2(\mathbb N)$ generated by 
\[
\Big((\rho_\infty\otimes I_{\ell_2(\mathbb N)}) \rtimes (t_\infty \otimes S)(\T_X^+)\Big)(\H\otimes e_1)\,.
\]
Then $(P_\K(\rho_\infty \otimes I_{\ell_2(\mathbb N)})P_\K, P_\K(t_\infty \otimes S)P_\K)$ is the minimal isometric dilation (coextension) of $(\rho_\infty\otimes I_2, t_\infty \otimes S_2)$. Moreover, the integrated form is a completely isometric representation of $\T_X^+$.
\end{lemma}
\begin{proof}
The first statement follows from \cite[Section 3]{MS1}. As there is a gauge action, the Gauge Invariant Uniqueness Theorem \cite[Theorem 6.2]{Katsura} implies that the integrated form of the compressed representation is a $*$-isomorphism of $\T_X$.
\end{proof}

\begin{lemma}
If $(X,\C)$ is a C*-correspondence, then 
\[
\K_n \ = \ \overline{\Span}\{(t_\infty(x_1)\otimes S_n)\dots(t_\infty(x_k)\otimes S_n)\H\otimes e_1 : x_1,\dots,x_k\in X, 0\leq k< n\}
\]
is a reducing subspace of $\rC^*((\rho_\infty\otimes I_n) \rtimes (t_\infty \otimes S_n)(\T_X^+))$. Moreover, by naturally embedding $\H \otimes \mathbb C^n$ into $\H \otimes \ell_2(\mathbb N)$, we see that 
\[
\overline{\bigcup_{n\geq 2} \K_n} = \K
\]
and 
\[
P_{\K_n}(\rho_\infty \otimes I_n) \rtimes (t_\infty \otimes S_n)P_{\K_n} = P_{\K_n}(\rho_\infty \otimes I_{\ell_2(\mathbb N)}) \rtimes (t_\infty \otimes S)P_{\K_n}\,.
\]
\end{lemma}
\begin{proof}
Clearly, $\K_n$ is invariant for $t_\infty(X)\otimes S_n$ as 
\[
(t_\infty(x_1)\otimes S_n)\dots(t_\infty(x_n)\otimes S_n) = 0
\]
for any $x_1,\dots, x_n\in X$. The conclusion that $\K_n$ is reducing comes from the following relationships, for any $c\in\C$ and $x_1,x_2\in X$,
\begin{align*}
(\rho_\infty(c)\otimes I_n)(t_\infty(x_1)\otimes S_n) & = t_\infty(\varphi(c)x_1)\otimes S_n, 
\\ (t_\infty(x_1)\otimes S_n)^*(t_\infty(x_2)\otimes S_n) & = \rho_\infty(\langle x_1,x_2\rangle)\otimes I_n,
\\ \rho_\infty(\C)\otimes I_n (\H\otimes e_1) & \subseteq \H\otimes e_1,\quad\text{and}
\\ (t_\infty(X)\otimes S_n)^*(\H\otimes e_1) & = 0\,.
\end{align*}
The last conclusion then follows immediately.
\end{proof}

We now have all the tools in place to produce meets of $\rC^*$-covers that fail to be residually finite-dimensional.

\begin{theorem}\label{t:fin-tens-meet}
Suppose $(X,\C)$ is a finite C*-correspondence with $\C$ unital and that there is $u\in X$ such that $\langle u,u\rangle = I$. 
If 
\[
\pi_n := P_{\K_n} (\rho_\infty  \otimes I_n) \rtimes (t_\infty  \otimes S_n)P_{\K_n}
\]
then
\[
 \ \bigwedge_{m\in \mathbb N} \left[ \rC^*\left(\left(\bigoplus_{n\geq m} \pi_n\right)(\T_X^+)\right), \ \bigoplus_{n\geq n} \pi_n\right] = \left[ \varinjlim_{m\rightarrow \infty} \rC^*\left(\left(\bigoplus_{n\geq m} \pi_n\right)(\T_X^+)\right), \ \pi\right] \
\]
is not an RFD C$^*$-cover of $\T_X^+$.
\end{theorem}
\begin{proof}
Let $(\rho,t)$ be the completely contractive representation of $(X,\C)$ such that $\pi = \rho\rtimes t$. In particular, we are interested in 
\[
t(u) = \varinjlim_{m\rightarrow \infty} \bigoplus_{n\geq m} P_{\K_n}(t_\infty(u) \otimes S_n) P_{\K_n}\,.
\]
Now $t_\infty(u)$ is an isometry since 
\[
t_\infty(u)^*t_\infty(u) = \rho_\infty(\langle u,u\rangle) = \rho_\infty(I) = I\,.
\]
This implies that $t(u)$ is a partial isometry since every $V_n = P_{\K_n}(t_\infty(u)  \otimes S_n)P_{\K_n}$ is a partial isometry. In what follows, we leverage the parital isometry $t(u)$ to construct an (infinite-dimensional) $\rC^*$-subalgebra of \[\varinjlim_{m\rightarrow \infty} \rC^*\left(\left(\bigoplus_{n\geq m} \pi_n\right)(\T_X^+)\right)\]that is $*$-isomorphic to the compact operators.

To this end, consider the C*-algebra
\[
\D_n \ := \ C^*(P_{\K_n}(t_\infty(x_1)\otimes S_n)(t_\infty(x_2)\otimes S_n)^* : x_1,x_2\in X),
\]
which is finite-dimensional since it is a subalgebra of $\rC^*(\pi_n(\T_X^+))$.
Upon observing that
\begin{align*}
(t_\infty(x_1)\otimes S_n)(t_\infty(x_2)\otimes S_n)^*&(t_\infty(y_1)\otimes S_n)(t_\infty(y_2)\otimes S_n)^*
\\ \ & = \ (t_\infty(x_1)\otimes S_n)(\rho_\infty(\langle x_2, y_1\rangle) \otimes (I - E_{nn}))(t_\infty(y_2)\otimes S_n)^*
\\ \ & = \ (t_\infty(x_1\langle x_2,y_1\rangle)\otimes S_n)(t_\infty(y_2)\otimes S_n)^*,
\end{align*}
we can conclude that 
\begin{align*}
\D_n & = \Span\{P_{\K_n}(t_\infty(x_1)\otimes S_n)(t_\infty(x_2)\otimes S_n)^* : x_1,x_2\in X\}
\\ & = P_{\K_n}\Big(\Span\{t_\infty(x_1)t_\infty(x_2)^* : x_1,x_2\in X\} \otimes (I_n - E_{11})\Big) =: P_{\K_n}(\D \otimes (I_n - E_{11}))\,.
\end{align*}
This means there exists $u_1,\dots,u_m,v_1,\dots,v_m \in X$ such that
\[
I_{\D} = \sum_{i=1}^m t_\infty(u_i)t_\infty(v_i)^*.
\]
As well, $I_{\D_n} = P_{\K_n}(I_\D \otimes (I_n - E_{11}))$.
Now let 
\[
P_n = P_{\K_n} - I_{\D_n} = P_{\K_n} - I_\D \otimes (I_n - E_{11}) \neq 0\,.
\]
Thus, 
\[
P = \rho(I) - \sum_{i=1}^m t(u_i)t(v_i)^* = \varinjlim_{m\rightarrow \infty} \bigoplus_{n\geq m} P_n \neq 0
\]
is a proper projection. 

Now
\begin{align*}
I_{\D_n}(t_\infty(y)\otimes S_n) & = I_{\D_n}(t_\infty(y)\otimes S_n)(I \otimes (I-E_{nn})
\\ & = \Big(I_{\D_n}(t_\infty(y)\otimes S_n)(t_\infty(u)\otimes S_n)^*\Big)(t_\infty(u)\otimes S_n)
\\ & = P_{\K_n}(t_\infty(y)\otimes S_n)(t_\infty(u)\otimes S_n)^*(t_\infty(u)\otimes S_n)
\\ & = P_{\K_n}(t_\infty(y)\otimes S_n)\,.
\end{align*}
Hence, $P_n(\H\otimes e_1)^\perp = (P_{\K_n} - I_{\D_n})(\H \otimes e_1)^\perp = 0$ and so $P(\H\otimes e_1)^\perp = 0$.

We are finally in a position to make our final calculations. By above $(t_\infty(u)\otimes S_n)^*P_n = 0$ and so 
\begin{align*}
t(u)^*P = 0\,.
\end{align*}
Additionally, for $1\leq k< n$,
\begin{align*}
(t_\infty(u)^*\otimes S_n^*)^k(t_\infty(u)\otimes S_n)^kP_n
& = (I \otimes (E_{11} + \dots + E_{n-k, n-k})P_n = P_n
\end{align*}
and so for all $k\geq 1$,
\[
(t(u)^*)^kt(u)^kP = P\,.
\]
This then gives us a system of matrix units $u_{ij} = t(u)^{i-1}P(t(u)^*)^{j-1} \in C^*(\pi(\T_X^+))$, as desired. Therefore, $C^*(\pi(\T_X^+))$ is not RFD as it has a subalgebra that is $*$-isomorphic to the compact operators.
\end{proof}

We now state the main result of this section.

\begin{theorem}\label{t:tens-lat}
Let $(X,\C)$ be a finite C*-correspondence where $\C$ is a unital C*-algebra. If there is $u\in X$ such that $\langle u,u\rangle = I$, then the space of residually finite-dimensional C*-covers for $\T_X^+$ do not form a complete lattice.
\end{theorem}
\begin{proof}
The conclusion follows from Proposition \ref{prop:finite} and Theorem \ref{t:fin-tens-meet}.
\end{proof}

We have focused on finite $\rC^*$-correspondences, although there are other classes of tensor algebras that are known to be RFD. For example, the tensor algebra of a (countable) graph correspondence is always RFD \cite[Theorem 4.7]{CR}. It is then conceivable that one could both replace the finiteness condition with a more general constraint and remove the unitality condition entirely. On the other hand, the constraint $\langle X,X\rangle = \C$ appears somewhat necessary as the conclusion of Theorem \ref{t:tens-lat} is not even true for the algebra of upper triangular $2\times 2$ matrices.

\section{The residually finite-dimensional C*-covers of the disc algebra}\label{s:disc-alg}

In the previous section, we saw that the residually finite-dimensional $\rC^*$-covers of an operator algebra do not always form a complete meet-semilattice. In this section, we offer a stronger conclusion for a particular choice of operator algebra: the RFD $\rC^*$-covers of the disc algebra $\mathrm{A}(\mathbb{D})$ do not even form a meet-semilattice. %In fact, we find two RFD $\rC^*$-covers for $\mathrm{A}(\mathbb{D})$ whose meet is not RFD.

On account of von Neumann's inequality \cite[Corollary 3.12]{Paulsen}, the representations of the disc algebra are in one-to-one correspondence with contractive Hilbert space operators. That is, whenever $T\in B(\H)$ is a contraction, we have a representation \[\rho_T:\mathrm{A}(\mathbb{D})\rightarrow B(\H), \quad f\mapsto f(T),\]defining a holomorphic functional calculus. In particular, a C*-cover $[\B, \iota]$ of the disc algebra $\mathrm{A}(\mathbb{D})$ has the property that $\B$ is a unital $\rC^*$-algebra generated by a single contraction $T$ and that $\iota(z) = T$. Furthermore, we note that a morphism of C*-covers $(\rC^*(T), j)\to (\rC^*(R),\kappa)$ corresponds to a $\ast$-homomorphism that maps $T$ to $R$. Thus, we make a few standing assumptions for this section. Firstly, when the context is clear, we eliminate the mention of the completely isometric embeddings $j$ and $\kappa$. In addition, the algebra $\rC^*(T)$ will denote the unital $\rC^*$-algebra generated by $T$, as we are dealing with C*-covers of the unital operator algebra $\mathrm{A}(\mathbb{D})$. %We will not encounter C*-covers of $\mathrm{A}(\mathbb{D})$ with the same underlying C*-algebra but two different embeddings in this section, so this hopefully won't cause any confusion.

\subsection{The C*-algebra of a punctured bilateral shift}

Let $\H=\ell^2(\bbN)$ with canonical orthonormal basis $\{e_1,e_2,e_3,\ldots\}$, and let $S\in B(\H)$ denote the unilateral shift defined by $Se_k=e_{k+1}$. The Toeplitz algebra $\rC^*(S)$ is a central candidate of a C*-cover for $\mathrm{A}(\mathbb{D})$ that is not RFD. Coburn's theorem also endows the Toeplitz algebra with a universal property: it is the universal C*-algebra generated by an isometry, and any representation of $\rC^*(S)$ mapping $S$ to a proper isometry is faithful. In this section, we derive an analogous universal property for the C*-algebra generated by the punctured bilateral shift $S\oplus S^\ast$.

\begin{remark}\label{rem:S+S*_not_RFD}
The C*-algebra $\rC^*(S\oplus S^\ast)$ contains all compact operators on $\H\oplus \H$, and so is not RFD. For instance, letting $P\in B(\H)$ denote the orthogonal projection onto $\bbC e_1$, we have
\begin{align*}
I-(S\oplus S^\ast)(S\oplus S^\ast)^\ast &=
P\oplus 0, \quad\text{and}\quad
I-(S\oplus S^\ast)^\ast (S\oplus S^\ast) =
0\oplus P
\end{align*}
are contained in $\rC^*(S\oplus S^\ast)$. A similar computation as for the Toeplitz algebra, where one applies the powers of $S\oplus S^\ast$ and its adjoint, guarantees that all matrix units for $\H\oplus\H$ lie in $\rC^*(S\oplus S^*)$ and thus, so do all compact operators on $\H\oplus\H$.
\end{remark}

The operator $S\oplus S^\ast$ is a partial isometry, but neither an isometry nor coisometry. However, $S\oplus S^*$ satisfies additional algebraic relations preventing the $\rC^*$-algebra it generates from being the universal C*-algebra generated by a partial isometry. For example, the domain and range space for $T = S\oplus S^*$ have orthogonal complements $0\oplus \bbC e_1$ and $\bbC e_1\oplus 0$ that are orthogonal. This constraint manifests itself in the algebraic relation
\[
(I-T^\ast T)(I-TT^\ast)=0.
\]
Furthermore, for any positive integers $m$ and $k$, one has $T^m(\bbC e_1\oplus 0)=\bbC e_{m+1}\oplus 0$ and $T^{\ast k} (0\oplus \bbC e_1)=0\oplus \bbC e_{k+1}$ also remain orthogonal. Algebraically, this condition translates to
\[
(T^{\ast k}(I-T^\ast T))^\ast(T^m(I-TT^\ast)) =
(I-T^\ast T)T^{m+k} (I-TT^\ast)=0.
\]This brings us to an appropriate analogue of Coburn's Theorem for the punctured bilateral shift. In what follows, we denote the range projection of a partial isometry $T$ by $p(T):=T T^\ast$, and the domain projection by $q(T):=T^\ast T$. Furthermore, we say that a partial isometry $T$ is \emph{proper} if it is nether an isometry nor a coisometry. %In particular, we recall that $p(T)T=T=Tq(T)$.

\begin{proposition}\label{prop:S+S*}
The C*-algebra $\rC^*(S\oplus S^\ast)$ is the universal unital C*-algebra generated by a partial isometry $T$ %=S\oplus S^\ast$ 
that satisfies
\begin{equation}\label{eq:S+S*_relation}
q(T)^\perp T^n p(T)^\perp = 0
\end{equation}
for all $n\ge 0$. Furthermore, any $\ast$-homomorphism of $\rC^*(S\oplus S^*)$ that maps $S\oplus S^*$ to a proper partial isometry is faithful.
\end{proposition}

%Not every partial isometry satisfies \eqref{eq:S+S*_relation}. For instance, we could post-compose $S\oplus S^\ast$ with a unitary that isn't of the form $U\oplus V$.

The proof of Proposition \ref{prop:S+S*} requires an extension of the Wold decomposition of an isometry. In fact, this recovers a special case of Halmos and Wallen's work on the decomposition of power partial isometries \cite{HW}. This follows as a partial isometry $T$ that satisfies Equation \eqref{eq:S+S*_relation} is automatically a \emph{power partial isometry}, in the sense that $T^n$ is a partial isometry for each $n\ge 1$. Indeed, upon expanding and rearranging Equation \eqref{eq:S+S*_relation},
\[
T^\ast T^{n+2} T^\ast =
T^\ast T^{n+1} + T^{n+1}T^\ast - T^n.
\]
Multiplying both sides by $T^{\ast(n+1)}$ gives
\begin{align*}
T^{\ast (n+2)} T^{(n+2)} T^{\ast (n+2)} &=
T^\ast (T^{\ast (n+1)}T^{(n+1)}T^{\ast (n+1)}) + 
 (T^{\ast (n+1)}T^{(n+1)}T^{\ast (n+1)})T^\ast \\&\qquad-
 T^\ast (T^{\ast n}T^nT^{\ast n})T^\ast.
\end{align*}
Applying induction on $n\ge 0$, it follows that $T^{\ast n}T^n T^{\ast n}=T^{\ast n}$ for all $n\ge 0$, and therefore $T$ is a power partial isometry. With that being said, as we do not require the full strength of Halmos and Wallen's work, we choose to include the details for a more self-contained proof.

%Recall that when $V\in B(\H)$ is an isometry, the classical Wold decomposition considers the wandering subspace $W=p(V)^\perp \H$, where $p(V)^\perp := I-p(V)$. One shows that the sum
%\[
%\bigoplus_{n\ge 0} V^n W
%\]
%is direct and reducing. Then, $V$ acts on $\bigoplus_{n\ge 0} V^nW$ as a power of the unilateral shift $S^{\oplus \dim(W)}$ (possibly infinite), and $V$ acts by a unitary on the orthogonal complement.

\begin{proposition}\label{prop:Wold_partial_iso}
Let $\H$ be any Hilbert space and let $T\in B(\H)$ be a partial isometry satisfying $q(T)^\perp T^n p(T)^\perp=0$ for every $n\ge 0$. Then there is a direct sum decomposition $\H=\K\oplus \J\oplus \L$ into reducing subspaces for $T$, with respect to which there is a unitary equivalence
\[
T\simeq
S^{\oplus \alpha} \oplus
(S^\ast)^{\oplus \beta} \oplus
U
\]
where $\alpha$ and $\beta$ are some cardinals (possibly vacuous or infinite), and $U$ is unitary (whenever $\L$ is non-zero).
\end{proposition}

\begin{proof}
Let $\W=p(T)^\perp \H$ and $\V = q(T)^\perp \H=p(T^\ast)^\perp \H$. Now, we set
\[
\K:=
\bigoplus_{n\ge 0} T^n \W,\qquad
\J:=
\bigoplus_{n\ge 0} T^{\ast n} \V,
\]
and let $\L:= (\K\oplus\J)^\perp$. For $m>n\ge 0$ and $x,y\in W$, we have
\[
\langle T^m x,T^ny\rangle =
\langle x,T^{\ast m}T^n(I-TT^\ast)y\rangle = 0
\]
upon applying the identity $T^\ast TT^\ast = T^\ast$. This then guarantees that the sum that defines $\K$ is, in fact, direct. A symmetric argument with $T^\ast$ shows that $\J$ also defines a direct sum.

We claim that $\K$ and $\J$ are mutually orthogonal reducing subspaces for $T$. To see orthogonality, let $m,n\ge 0$ be positive integers, $x\in \W$, and $y\in \V$. Then we have
\[
\langle T^mx,T^{\ast n} y\rangle =
\langle T^{m+n}x,y\rangle =
\langle T^{m+n}p(T)^\perp x,q(T)^\perp y\rangle = 0
\]
as $q(T)^\perp T^{m+n}p(T)^\perp=0$, which guarantees that $\K$ and $\J$ are orthogonal. To see that $\K$ and $\J$ are reducing for $T$, first note that $\K$ is clearly $T$-invariant. To see that it is $T^\ast$-invariant, let $x\in \W$ and $n\ge 0$. Then
\[
T^\ast T^n x=
T^\ast T^n(I-TT^\ast)x.
\]
If $n=0$, the identity $T^\ast = T^\ast TT^\ast$ implies that $T^\ast x=0$. If $n\ge 1$, then
\begin{align*}
T^\ast T^n x &= 
(T^\ast T)T^{n-1}p(T)^\perp x \\ &=
(I-q(T)^\perp)T^{n-1}p(T)^\perp x \\ &=
T^{n-1}p(T)^\perp x \in \K
\end{align*}
because $q(T)^\perp T^{n-1}p(T)^\perp = 0$ by assumption. That is, $\K$ is $T^\ast$-invariant as desired. A symmetric argument with $T^\ast$ in place of $T$ shows that $\J$ is reducing for $T^\ast$ and thus reducing for $T$ as well.

Finally, as in the proof of the classical Wold decomposition, one has that $T\vert_\K$ is an isometry and that $\W$ is a wandering subspace for $T\vert_\K$. In particular, we have $T\vert_\K\simeq S^{\oplus \alpha}$ where $\alpha=\dim(\W)$. Symmetrically, the same argument for $T^\ast$ on $\J$ shows that $T^\ast\vert_\J\simeq S^{\oplus \beta}$ where $\beta=\dim(\V)$, and that $T\vert_\J\simeq (S^\ast)^{\oplus \beta}$. Lastly, it automatically follows that $\L=(\K\oplus \J)^\perp$ is reducing for $T$ and, as both $p(T)^\perp=I-TT^\ast$ and $q(T)^\perp = I-T^\ast T$ restrict to the zero operator on $\L$, we must have $U:=T\vert_\L$ is unitary whenever $\L$ is non-zero.
\end{proof}

As mentioned above, Proposition \ref{prop:Wold_partial_iso} can instead be recovered using the main result in \cite{HW}, whereby any power partial isometry can be decomposed into a direct sum of a unitary, $S$, $S^\ast$ and truncated shifts $P_{\bbC^m}S\vert_{\bbC^m}$, with multiplicities. From this, it suffices to show that the truncated shifts do not satisfy Equation \eqref{eq:S+S*_relation}. This is a direct calculation and, in fact, a keen-eyed reader will see this occur below, embedded in the proof of Proposition \ref{prop:infinite_compression}.

\begin{proof}[Proof of Proposition \ref{prop:S+S*}]
Suppose that $T$ is a partial isometry satisfying the relation \eqref{eq:S+S*_relation} for every $n$. Proposition \ref{prop:Wold_partial_iso} shows that, up to a unitary equivalence, $T=S^{\oplus \alpha}\oplus (S^\ast)^{\oplus \beta}\oplus U$ for some cardinals $\alpha$ and $\beta$ and a unitary $U$.

Since $C(\bbT)=\rC^*(z)$ is the C*-envelope of $\mathrm{A}(\mathbb{D})$, there is a morphism $\rC^*(S\oplus S^\ast)\to C(\bbT)$ of $\rC^*$-covers mapping $S\oplus S^\ast$ to $z$. Furthermore, continuous functional calculus yields a $\ast$-homomorphism $C(\bbT)\to \rC^*(U)$ mapping $z$ to $U$, which then provides us with a $\ast$-homomorphism $\pi:\rC^*(S\oplus S^\ast)\to \rC^*(U)$ such that $\pi(S\oplus S^\ast)=U$. Thus, we have a $\ast$-homomorphism
\[
\rC^*(S\oplus S^\ast)\xrightarrow{(\id^{\oplus \alpha}\oplus \id^{\oplus \beta})\oplus \pi} \rC^*(S^{\oplus \alpha}\oplus (S^\ast)^{\oplus \beta}\oplus U)
\]
where $S\oplus S^\ast$ is mapped to $T$. In particular, this is also a morphism of $\rC^*$-covers for $\mathrm{A}(\bbD)$.

Now, assume that $T$ is a proper partial isometry, which occurs if and only if $\alpha\ne 0$ and $\beta\ne 0$. Then, the $\ast$-homomorphism
\[
\rC^*(S^{\oplus \alpha}\oplus (S^\ast)^{\oplus \beta}\oplus U)\longrightarrow \rC^*(S\oplus S^\ast)
\]
that maps $T$ to $S\oplus S^\ast$ is an inverse to $\id^{\oplus \alpha}\oplus \id^{\oplus \beta}\oplus U$. Therefore, the unique homomorphism $\rC^*(S\oplus S^\ast)\to \rC^*(T)$ that maps $S\oplus S^\ast$ to $T$ is faithful.
\end{proof}

\subsection{The semilattice of RFD C*-covers for the disc algebra}

As above, let $S$ be the unilateral shift on $\H=\overline{\text{span}}\{e_1,e_2,e_3,\ldots\}$. In addition, for $n\ge0$, set $\H_n=\text{span}\{e_1,\ldots,e_n\}\simeq \bbC^n$ and let $P_n$ be the projection onto $\H_n$. In this section, we consider $\rC^*$-covers for $\mathrm{A}(\bbD)$ arising from the truncated unilateral shifts
\[
S_n:=
P_n S\vert_{\H_n}= 
\begin{bmatrix} 
    0 & & & 0 \\
    1 & 0 \\ 
    &\ddots&\ddots \\ 
    0 &&1&0 
\end{bmatrix}.
\]
Since $\H_n$ is coinvariant for $S$, the compression map $B(\H)\rightarrow M_n$ restricts to a representation of $\overline{\text{alg}}(I,S)\simeq\mathrm{A}(\mathbb{D})$. Furthermore, the $\rC^*$-algebra generated by the infinite direct sum
\[
\rC^*\left(\bigoplus_{n\ge 1} S_n\right)\subseteq
\prod_{n\ge 1}M_n
\]
gives rise to a C*-cover of $\mathrm{A}(\mathbb{D})$, which is RFD by design. In fact, the same is true if we consider the direct sum of the truncated shift operators against any increasing subsequence of $\bbN$.

%However, before continuing, we remark that $S_n$ generates $M_n$ as a C*-algebra. Indeed, this may be seen as a finite-dimensional analogue to the Toeplitz algebra containing all compact operators. For example, $S_n^{n-1}$ is the matrix unit $E_{n1}$ and $S_n^{\ast(n-1)}S_n^{n-1} = E_{11}$. Then, any matrix unit lies in $\rC^*(S_n)$ by multiplying by $S_n$ and $S_n^\ast$. An explicit formula for this is given by
%\[
%E_{ij}=S_n^{(i-1)}(S_n^{\ast(n-1)}S_n^{(n-1)})S_n^{\ast(j-1)}, \qquad 1\le i,j\le n.
%\]

\begin{proposition}\label{prop:infinite_compression}
Let $X$ be an infinite subset of $\bbN$, and consider the C*-algebra
\[
\rC^*\left(\bigoplus_{n\in X} S_n\right) \subseteq
\prod_{n\in X} M_n.
\]Then the following statements hold.
\begin{itemize}
\item [(1)] The C*-algebra $\rC^*(\bigoplus_{n \in X} S_n)$ is an RFD C*-cover of $\mathrm{A}(\mathbb{D})$.
\item [(2)] The ideal
\[
\bigoplus_{n\in X} M_n
\]
is contained in $\rC^*(\bigoplus_{n \in X} S_n)$.
\item [(3)] The ideal $\bigoplus_{n \in X}M_n$ is a boundary ideal for $\mathrm{A}(\mathbb{D})$. Furthermore, we have
\[
\frac{\rC^*\left(\bigoplus_{n\in X} S_n\right)}{\bigoplus_{n\in X}M_n} \simeq
\rC^*(S\oplus S^\ast)
\]
via a generator-preserving $*$-isomorphism. In particular, these C*-algebras determine equivalent C*-covers for $\mathrm{A}(\bbD)$.
\end{itemize}
\end{proposition}

\begin{proof}
(1) By construction, $\rC^*(\bigoplus_{n \in X} S_n)$ is RFD. Furthermore, $\rC^*(S\oplus S^\ast)$ is a C*-cover for $\mathrm{A}(\mathbb{D})$ and so, the remainder of (1) will automatically follow from (3).

\bigskip

(2) It suffices to show that (2) holds in the case $X=\bbN$. Indeed, if this were true, the coordinate projection $\prod_{n\in \bbN} M_n\to \prod_{n\in X}M_n$ would restrict to a morphism of C*-covers \[\rC^*\left(\bigoplus_{n\in \bbN} M_n\right)\longrightarrow \rC^*\left(\bigoplus_{n\in X}M_n\right)\] that maps $\bigoplus_{n\in \bbN} M_n$ surjectively onto $\bigoplus_{n\in X}M_n$.

Accordingly, to show that $\bigoplus_{n\in \bbN} M_n\subseteq \rC^*(\bigoplus_{n\in \bbN} S_n)$, it suffices to show that $\rC^*(\bigoplus_{n\in \bbN} S_n)$ contains the central projections
\[
Q_k =
0\oplus \cdots \oplus 0 \oplus I_k\oplus 0\oplus \cdots,
\]
where the identity matrix appears in the $k$-th summand. Indeed, if this were the case, then the C*-algebra $\rC^*(\bigoplus_{n\in\bbN}S_n)$ would contain
\[
Q_n\left(\bigoplus_{n\in \bbN} S_n\right) =
0\oplus \cdots \oplus 0 \oplus S_k\oplus 0\cdots,\]
and, since $\rC^*(S_n) = M_n$, we may conclude that $\bigoplus_{n\in \bbN} M_n$ is contained in $\rC^*(\bigoplus_{n\in \bbN} S_n)$.

To this end, we may express $S_n$ using the standard matrix units as
\[
S_n =
E_{21}+E_{32}+\cdots+E_{n,(n-1)}.
\]
Upon considering how this expression acts with respect to standard basis vectors in $\bbC^n$, one can directly verify that $S_n^{\ast k}S_n^k$ is the diagonal matrix
\[
S_n^{\ast k}S_n^k =
\text{diag}(1,\ldots,1,0,\ldots,0)
\]
with $\max\{n-k,0\}$ non-zero entries for any $k,n\in\bbN$. When $k< n$, note that
\begin{align*}
S_n^{\ast k}S_n^k &=
\text{diag}(1,\ldots,1,0,0,\ldots,0), \\
S_n(S_n^{\ast k}S_n^k)S_n^\ast &=
\text{diag}(0,1,\ldots,1,0,\ldots,0), \\
&\,\vdots \\
S_n^{k-1}(S_n^{\ast k}S_n^k)S_n^{\ast (k-1)} &=
\text{diag}(0,0,\ldots,0,1,\ldots,1)
\end{align*}
with $n-k$ non-zero entries for each matrix. When $k<n$, we then have that the matrix
\[
\sum_{j=0}^{k-1}S_n^j (S_n^{\ast k}S_n^k)S_n^{\ast j}
\]
is diagonal with positive integer entries. Otherwise, if $k \ge n$, then this matrix is just the zero matrix, because $S_n^n=0$.

Thus far, we have established that for the $\ast$-polynomial
\[
f_k(z) =
\sum_{j=0}^{k-1} z^j(z^{\ast k}z^k)z^j,
\]
and the operator $\bigoplus_{n\in \bbN} S_n$, we have that
\[
f_k\left(\bigoplus_{n\in \bbN} S_n\right) =
0\oplus \cdots \oplus 0 \oplus f_k(S_{k+1})\oplus f_k(S_{k+2})\oplus\cdots,
\]
where there are $n$ zero summands and, for $n>k$, the matrix $f_k(S_n)$ is diagonal with diagonal entries at least $1$. Now, if $g:[0,\infty)\to [0,\infty)$ is a continuous function with the properties that $g(0)=0$ and $g(t)=1$ for $t\ge 1$, then
\[
g\left(f_k\left(\bigoplus_{n\in \bbN} S_n\right)\right) =
0\oplus \cdots \oplus 0 \oplus I_{k+1}\oplus I_{k+2}\oplus \cdots 
\]
lies in $\rC^*(\bigoplus_{n\in \bbN} S_n)$ for each $k\in\bbN$. Therefore, upon subtracting, we have $\rC^*(\bigoplus_{n\in \bbN}S_n)$ also contains
\[
Q_k =
0\oplus \cdots 0 \oplus I_k\oplus 0 \oplus \cdots
\]
for each $k\in\bbN$, as desired.

\bigskip

(3) Throughout, we set \[T=\bigoplus_{n\in X} S_n +\bigoplus_{n\in X} M_n,\]which is the coset of $\bigoplus_{n\in X}S_n$ inside $\rC^*(\bigoplus_{n\in X}S_n)/(\bigoplus_{n\in X} M_n)$. We claim that $\rC^*(S\oplus S^\ast)\simeq \rC^*(T)$ by a $\ast$-isomorphism that maps $S\oplus S^\ast$ to $T$. This will be accomplished by proving that $T$ is a proper partial isometry satisfying the relations acquired in Proposition \ref{prop:S+S*}.

We first verify that $T$ is a proper partial isometry. That $T$ is a partial isometry follows from $S_n$ being a partial isometry for each $n\in\bbN$. Furthermore, $S_n^\ast S_n=I_n-E_{nn}$ for each $n\ge 2$, so that $T^\ast T$ is not the identity operator. Similarly, $S_nS_n^\ast=I_n-E_{11}$ for all $n\ge 2$, which implies that $TT^\ast\ne I$ and that $T$ must be proper.

It remains to show that
\[
(I-T^\ast T)T^k(I-T^\ast T) =0
\]
for every $k\ge 0$. For each $n\in\bbN$ such that $n\neq k$, the truncated shifts $S_n$ satisfies
\begin{align*}
(I-S_n^\ast S_n)S_n^k(I-S_nS_n^\ast) &=
E_{nn}S_n^kE_{11} =
\begin{cases}
    E_{nn}E_{k1}, & k\le n \\
    0, & k\ge n.
\end{cases}
\end{align*}
Thus, upon passing to the quotient by $\bigoplus_{n \in X}M_n$, we have $(I-T T^\ast)T^k(I-T^\ast T)=0$. Therefore, by Proposition \ref{prop:S+S*}, there is a $\ast$-isomorphism $\rC^*(S\oplus S^\ast)\to \rC^*(T)$ mapping $S\oplus S^\ast$ to $T$, which then induces an equivalence of C*-covers for $\mathrm{A}(\mathbb{D})$.

Finally, note that the quotient mapping \[
\rC^*\left(\bigoplus_{n\in X}S_n\right)\longrightarrow\frac{\rC^*\left(\bigoplus_{n\in X} S_n\right)}{\bigoplus_{n\in X}M_n}
\]is a surjective $*$-homomorphism that evidently maps $\bigoplus_{n\in X} S_n$ to $T$. Thus, we have that $\rC^*\left(\bigoplus_{n\in X}S_n\right)$ must in fact be a $\rC^*$-cover for $\mathrm{A}(\bbD)$ and that $\bigoplus_{n\in X}M_n$ is a boundary ideal.
\end{proof}

We will now need to compare the RFD C*-covers
\[
\rC^*\left(\bigoplus_{n\in \bbN} S_n\right) \quad\text{and}\quad
\rC^*\left(\bigoplus_{n\in Y} S_n\right),
\]
where $X\subset \bbN$ is some infinite subset. In particular, we note that these determine distinct C*-covers for $\mathrm{A}(\bbD)$.

\begin{proposition}\label{prop:infinite_compression_morphism}
Let $X\subseteq \bbN$ be an infinite subset. Then the projection map $\prod_{n\in \bbN} M_n \to \prod_{n\in X} M_n$ restricts to a morphism
\[
\pi:\rC^*\left(\bigoplus_{n\in \bbN} S_n\right) \to
\rC^*\left(\bigoplus_{n\in X} S_n\right)
\]
of C*-covers of $\mathrm{A}(\mathbb{D})$, with corresponding boundary ideal given by $\bigoplus_{n\in X^c} M_n$.
\end{proposition}

\begin{proof}
Since $\pi$ is given by coordinate projection, it is immediate that $\pi$ is a morphism of $\rC^*$-covers, and that $\ker\pi\supseteq \bigoplus_{n\in X^c} M_n$.

To establish the reverse inclusion, note that Proposition \ref{prop:infinite_compression} implies that there is a morphism of $\rC^*$-covers
\[
\sigma:
\rC^*\left(\bigoplus_{n\in X} S_n\right)\to 
\rC^*(S\oplus S^\ast).
\] 
Then $\sigma\pi:\rC^*(\bigoplus_{n\in \bbN} S_n)\to \rC^*(S\oplus S^\ast)$ is a C*-cover morphism. As morphisms of $\rC^*$-covers are unique, by Proposition \ref{prop:infinite_compression} (3), we must have $\ker(\sigma\pi)=\bigoplus_{n\in \bbN} M_n$. Since $\ker \pi\subseteq \ker (\sigma\pi)$, we have established that
\[
\bigoplus_{n\in X^c} M_n\subseteq
\ker\pi\subseteq \bigoplus_{n\in \bbN} M_n.
\]
Now, since $\pi$ is the coordinate projection and $\rC^*(\bigoplus_{n\in X} S_n)$ contains $\bigoplus_{n\in X} M_n$ by Proposition \ref{prop:infinite_compression} (2), it must be that $\pi$ acts faithfully on $\bigoplus_{n\in X} M_n$. Therefore, we find that $\ker\pi$ is contained in $\bigoplus_{n\in X^c}M_n$ as desired.
\end{proof}

%We should note that Proposition \ref{prop:infinite_compression_morphism} extends readily to a subset of a subset $X\subseteq Y\subseteq \bbN$, with the same proof, but we don't need this level of generality.

We have now performed all the necessary computations to show that the RFD C*-covers of $\mathrm{A}(\mathbb{D})$ are not closed under finite meets. For this, recall that to describe the meet of two C*-covers $[\B,\iota]$ and $[\C,\eta]$, it is sufficient to consider any C*-cover $[\D, \rho]$ such that $[\D,\rho]\ge [\B,\iota],[\C,\eta]$. Indeed, this follows from Remark \ref{r:meet} because if $\pi:(\D,\rho)\to (\B,\iota)$ and $\sigma:(\D,\rho)\to (\C,\eta)$ are the unique C*-cover morphisms, then the meet is given by
\[
[\B,\iota]\wedge[\C,\eta] = \left[\frac{\D}{(\overline{\ker\pi+\ker\sigma})},q\rho\right],
\]
where $q:\D\to \D/(\overline{\ker\pi+\ker\sigma})$ is the quotient map. This presents us with the main result of this section.

\begin{theorem}\label{t:disc-lattice}
    The following statements hold.\begin{enumerate}
        \item[(1)] There exist two RFD $\rC^*$-covers $[\B, \iota], [\C, j]$ for $\mathrm{A}(\bbD)$ such that the meet $[\B, \iota]\wedge[\C, j]$ is not an RFD C*-cover for $\mathrm{A}(\bbD)$.
        \item[(2)] There exists a (totally ordered) chain of distinct RFD $\rC^*$-covers $[\B_k, \iota_k]$ for $\mathrm{A}(\bbD)$ whose meet is not an RFD C*-cover.
    \end{enumerate}In particular, the space of residually finite-dimensional C*-covers for $\mathrm{A}(\bbD)$ is not a lattice.
\end{theorem}

\begin{proof}
    (1) Let $E\subseteq \bbN$ be the set of even natural numbers, so that the complement $E^c$ is the set of odd naturals. By Proposition \ref{prop:infinite_compression} (1), both $\rC^*(\bigoplus_{n\in E}S_n)$ and $\rC^*(\bigoplus_{n\in E^c} S_n)$ produce RFD C*-covers of $\mathrm{A}(\mathbb{D})$. We claim that the meet of these $\rC^*$-covers for $\mathrm{A}(\mathbb{D})$ is $\rC^*(S\oplus S^*)$.

    Indeed, note that the coordinate projections are morphisms of $\rC^*$-covers
    \[\begin{tikzcd}
       & \rC^*\left(\bigoplus_{n\in \bbN}M_n\right) \arrow[ld,swap,"\pi_1"] \arrow[rd,"\pi_2"] & \\
        \rC^*\left(\bigoplus_{n\in E^c}M_n\right) & & \rC^*\left(\bigoplus_{n\in E}M_n\right)
    \end{tikzcd}\]
    and so the meet of these $\rC^*$-covers is
    \[
    \frac{\rC^*\left(\bigoplus_{n\in \bbN}M_n\right)}{(\overline{\ker\pi_1+\ker\pi_2})}.
    \]
    By Proposition \ref{prop:infinite_compression_morphism}, we have
    \[
    \ker \pi_1 = \bigoplus_{n\in E} M_n \quad\text{and}\quad
    \ker\pi_2 = 
    \bigoplus_{n\in E^c} M_n.
    \]
    In particular, we must have that $\ker\pi_1+\ker\pi_2$ is equal to $\bigoplus_{n\in \bbN} M_n$, which is also closed. By Proposition \ref{prop:infinite_compression} (3), we then have an equivalence of $\rC^*$-covers
    \[
    \frac{\rC^*\left(\bigoplus_{n\in \bbN}M_n\right)}{(\overline{\ker\pi_1+\ker\pi_2})} =
    \frac{\rC^*\left(\bigoplus_{n\in \bbN}M_n\right)}{\bigoplus_{n \in \bbN} M_n} \simeq
    \rC^*(S\oplus S^\ast),
    \]
    and $\rC^*(S\oplus S^*)$ not RFD (Remark \ref{rem:S+S*_not_RFD}). In particular, the space of RFD C*-covers for $\mathrm{A}(\bbD)$ do not form a lattice.

    (2) Consider the totally ordered chain of RFD C*-covers
    \[
    \rC^*\left(\bigoplus_{n\ge 1}S_n\right)\to 
    \rC^*\left(\bigoplus_{n\ge 2}S_n\right) \to
    \rC^*\left(\bigoplus_{n\ge 3}S_n\right) \to\cdots
    \]
    of $\mathrm{A}(\mathbb{D})$, where the maps are the natural projections. By Proposition \ref{p:meet}, the meet of this chain is the direct limit
    \[
    \varinjlim_k \rC^*\left(\bigoplus_{n\ge k} S_n\right)
    \]
    taken along the sequence of quotient mappings. Up to equivalence of $\rC^*$-covers, the directed system of $\rC^*$-covers agrees with
    \[
    \rC^*\left(\bigoplus_{n\ge 1}S_n\right)\to
    \frac{\rC^*\left(\bigoplus_{n\ge 1}S_n\right)}{M_1} 
    \to\frac{\rC^*\left(\bigoplus_{n\ge 1}S_n\right)}{M_1\oplus M_2}
    \to\frac{\rC^*\left(\bigoplus_{n\ge 1}S_n\right)}{M_1\oplus M_2\oplus M_3}\to \cdots.
    \]
    It is then straightforward to check that the direct limit is
    \[
    \varinjlim_k \rC^*\left(\bigoplus_{n\ge k} S_n\right) =
    \frac{\rC^*\left(\bigoplus_{n\ge 1}S_n\right)}{\bigoplus_{n\ge 1} M_n},
    \]
    which, by Proposition \ref{prop:infinite_compression} (3), is equivalent to $C(S\oplus S^\ast)$ as a $\rC^*$-cover of $\mathrm{A}(\bbD)$.
\end{proof}

%Note that part (2) of the previous theorem was already established by Theorem \ref{t:fin-tens-meet} since $A(\mathbb D)$ is completely isometrically isomorphic to $\T_{(\mathbb C,\mathbb C)}^+$.

%\begin{remark}
%    There is nothing special about the choice of RFD $\rC^*$-covers considered in Theorem \ref{t:disc-lattice} (i). In particular, one could use any infinite subset of $\bbN$ and its complement. %\textbf{[In fact, I think any two infinite subsets whose intersection is finite do the trick.]}
%\end{remark}

\section{A couniversal residually finite-dimensional C*-cover}\label{s:rfd-env}

In this section, we answer a question raised in \cite{HumRam}. Given a residually finite-dimensional operator algebra $\A$, is the meet of all RFD $\rC^*$-covers for $\A$ still an RFD $\rC^*$-cover? A positive answer would produce a couniversal RFD $\rC^*$-cover for $\A$, effectively resulting in an RFD replacement for the $\rC^*$-envelope for $\A$.

There are many classes of operator algebras where a positive answer to the problem can be established. For example, the meet of all RFD $\rC^*$-covers of a so-called completely isometrically subhomogeneous operator algebra is the $\rC^*$-envelope, and must necessarily be RFD \cite[Proposition 4.1]{AHMR}. A multitude of other constraints guarantee that the $\rC^*$-envelope is RFD \cite[Section 6]{CR}, exhibiting a positive, yet trivial, answer to the problem at hand. Nevertheless, this is by no means exhaustive, as we now elaborate. The remainder of this section will highlight several examples of RFD operator algebras that do not possess a minimal RFD $\rC^*$-cover. Our first example takes the form of a tensor algebra of a non-finite $\rC^*$-correspondence.

\begin{theorem}\label{t:tri-compact}
Let $\A\subset B(\ell^2(\bbN))$ denote the subalgebra of upper triangular compact operators. Then the following statements hold.\begin{enumerate}
    \item[(1)] The meet of all residually finite-dimensional C*-covers for $\A$ is $K(\ell^2(\bbN))$.
    \item[(2)] The space of residually finite-dimensional C*-covers for $\A$ does not have a minimal element.
\end{enumerate}
\end{theorem}

\begin{proof}
We first remark that (2) follows from (1) because the simplicity of $K(\ell^2(\bbN))$ forces it to be the $\rC^*$-envelope for $\A$. Furthermore, note that $\A$ is completely isometrically isomorphic to the tensor algebra of the $\rC^*$-correspondence $(X,\C) = (c_0, c_0)$ where $\C$ is the diagonal and $X$ is the superdiagonal in $K(\ell^2(\mathbb N))$.

Let $V_n$ be the natural isometry $V_n: \spn\{e_1,\dots, e_n\} \rightarrow \ell^2(\mathbb N)$. Since $\A$ consists of upper triangular operators, the compression map $\rho_n:\A\rightarrow M_n$ defined by $\rho_n(A)= V_n^* AV_n$ is a representation of $\A$. It is then straightforward to see that
\begin{align*}
\bigwedge_{m\geq 1} \left[ C^*\left(\left(\bigoplus_{n\geq m} \rho_n\right)(\A)\right), \bigoplus_{n\geq m} \rho_n\right] 
& \ \ = \ \  \left[ C^*\left(\left(\varinjlim_{m\rightarrow \infty} \bigoplus_{n\geq m} \rho_n\right)(\A)\right), \varinjlim_{m\rightarrow \infty} \bigoplus_{n\geq m} \rho_n\right]
\\ \\ &\ \ = \ \ \left[ K(\ell^2(\mathbb N), \iota_{\min} \right]\,.
\end{align*}
For each $k,m\in\bbN$ such that $m\geq k$, the above equivalence embeds the algebra of upper triangular matrices, $T_k\subset M_k$, with identity and zero representations: 
\[
C^*\left(\left(\bigoplus_{n\geq m} \rho_n\right)(V_k T_k V_k^*) \right) \simeq M_k\,.
\] Since the compact operators $K(\ell^2(\bbN))$ are not RFD, we obtain the desired conclusion.
\end{proof}

Next, we show that the RFD C*-covers for the so-called standard embedding TUHF algebra do not form a lattice. For this, recall that the {\bf standard embedding} of $M_n$ into $M_{kn}$ is 
\[
A \ \mapsto \ A \oplus A \oplus \cdots \oplus A = I_k \otimes A
\]
and the {\bf refinement embedding} is 
\[
A \ \mapsto \ A\otimes I_k\,.
\]
When considering the algebra $T_2$, the standard and refinement embeddings are
\[
\left[\begin{matrix} a&b\\ &c \end{matrix}\right] \mapsto \left[\begin{matrix} a&b\\ &c  \\ && a& b \\ && &c \end{matrix}\right] \quad \textrm{and} \quad 
\left[\begin{matrix} a&b\\ &c \end{matrix}\right] \mapsto \left[\begin{matrix} a&&b\\ &a&&b \\ &&c \\ &&&c \end{matrix}\right],
\]
respectively. The standard embedding $2^\infty$ TUHF algebra, denoted by $\T_{s,2^\infty}$, is the direct limit of $T_{2^n}, n\geq 1,$ with respect to standard embeddings. Similarly, the refinement embedding $2^\infty$ TUHF algebra, denoted $\T_{r,2^\infty}$, is the direct limit of $T_{2^n}, n\geq 1,$ along the refinement embeddings.

These two operator algebras are not completely isometrically isomorphic, and there are many methods to identify why this is the case. For one, the standard embedding $2^\infty$ TUHF algebra is completely isometrically isomorphic to a tensor algebra of a C*-correspondence, while the refinement is not \cite{KatRamLimit}. Furthermore, one can observe that $\T_{s,2^\infty}$ is residually finite-dimensional, yet $\T_{r,2^\infty}$ is not. The former is discussed in further detail below, while for the latter, one can see that every off-diagonal matrix unit in $T_{2^n}$ has a form of infinite divisibility:
\[
T_2 \ni e_{1,2}^{(1)}  = e_{1,2}^{(2)}e_{2,3}^{(2)} + e_{2,3}^{(2)}e_{3,4}^{(2)} \in T_4\,.
\]
Among other things, this implies that every finite-dimensional representation of $\T_{r,2^\infty}$ is zero on the off-diagonal. Consequently, this can only define a representation of the diagonal $\T_{r,2^\infty}\cap \T_{r,2^\infty}^*$ of $\T_{r,2^\infty}$, which is the space of continuous functions on the Cantor set.

\begin{theorem}
Let $\T_{s, 2^\infty}$ denote the standard embedding $2^\infty$ TUHF algebra. Then the following statements hold. \begin{enumerate}
    \item [(1)] The meet of all residually finite-dimensional C*-covers for $\T_{s, 2^\infty}$ is $M_{2^\infty}$.
    \item [(2)] The space of residually finite-dimensional C*-covers for $\T_{s,2^\infty}$ does not have a minimal element.
\end{enumerate}
\end{theorem}

\begin{proof}
As in Theorem \ref{t:tri-compact}, statement (2) will immediately follow from (1) as $M_{2^\infty}$ is simple and infinite-dimensional. Now, for each $n\geq 1$, we embed $T_{2^n}$ in $B(\ell^2(\mathbb N))$ using an infinite standard embedding 
\[
\psi_n : A \ \mapsto \ A\oplus A \oplus A \oplus \dots\,.
\]
Let $\varphi_n :T_{2^n} \rightarrow T_{2^{n+1}}$ denote the standard embedding. It is immediate that $\psi_{n+1}\varphi_n = \psi_n$ and thus, there is a completely isometric isomorphism
\[
\psi : \T_{s,2^\infty} \rightarrow \overline{\bigcup_{n\geq 1} \psi_n(T_{2^n})}.
\]
Moreover, since $C^*(\psi_n(T_{2^n})) \simeq M_{2^n}$, we have $C^*(\psi(\T_{s,2^\infty})) \simeq M_{2^\infty}$.

Define a completely positive map $\rho_n : B(\ell^2(\mathbb N)) \rightarrow M_n$ by $\rho_n(A) = V_n^*AV_n$. Furthermore, note that $\rho_n$ is multiplicative on the algebra of upper triangular operators. Thus, for each $m,k\in\bbN$ such that $m\geq k$, we have
\[
C^*\left(\left(\bigoplus_{n\geq m} \rho_{2^n} \right)(\psi_n(T_{2^n}))\right) \simeq M_{2^n}\,.
\]
Hence,
\begin{align*}
\bigwedge_{m\geq 1} \left[ C^*\left(\left(\bigoplus_{n\geq m} \rho_{2^n}\right)(\T_{s,2^\infty})\right), \bigoplus_{n\geq m} \rho_{2^n}\right] 
& \ \ = \ \  \left[ C^*\left(\left(\varinjlim_{m\rightarrow \infty} \bigoplus_{n\geq m} \rho_{2^n}\right)(\T_{s,2^\infty})\right), \varinjlim_{m\rightarrow \infty} \bigoplus_{n\geq m} \rho_{2^n}\right]
\\ \\ &\ \ = \ \ \left[ M_{2^\infty}, \iota_{\min} \right]\,. \qedhere
\end{align*}
\end{proof}

We note that Larson and Solel have previously studied residual finite-dimensionality for triangular AF algebras \cite{LarsonSolel}. However, their definition of residual finite-dimensionality is that the algebra is merely separated by its finite-dimensional representations. At present, the connection between the notion of residual finite-dimensionality we study here and separation by finite-dimensional representations remains a somewhat mysterious phenomenon \cite[Section 3]{Hartz}.

%{\color{red} We could add other examples, triangular limit algebras that aren't semi-Dirichlet, if it is desired. }

We consider one final example showing that the $\rC^*$-correspondence may be taken to be finite and that the operator algebra may also be taken to be unital. To this end, we record some finer details surrounding the $\rC^*$-correspondence of Example \ref{e:popescu}.

For each $d\in\bbN$, let $\bbF_d^+$ denote the free monoid generated by $\{1,\ldots, d\}$. For a word $w\in\bbF_d^+$, we let $|w|$ denote the length of the word. Furthermore, we let $\fF_d^2$ denote the Hilbert space $\ell^2(\bbF_d^+)$ and $\{e_w : w\in\bbF_d^+\}$ denote the canonical orthonormal basis for $\fF_d^2$. Construct a $d$-tuple of operators $L = (L_1, \ldots, L_d)$ on $\fF_d^2$ by prescribing \[L_je_w = e_{jw}, \ \ \ \ \ \ \ \ w\in\bbF_d^+,\  j=1,\ldots, d.\]It can be easily verified that $\{L_1\ldots, L_d\}$ is a family of isometries with pairwise orthogonal ranges and that, when $d=1$, we recover the classical unilateral shift. The unital norm-closed algebra $\A_d$ that is generated by $\{L_1, \ldots, L_d\}$ is called the {\bf non-commutative disc algebra}. If we let $\T_d:=\rC^*(\A_d)$, then $\T_d$ is the Cuntz--Toeplitz algebra. In particular, there is a short exact sequence\[0\longrightarrow K(\fF_d^2)\longrightarrow\T_d\longrightarrow\O_d\longrightarrow0\]where $\O_d$ denotes the Cuntz algebra.

As we have noted, the operator algebra $\A_d$ can be recognized as the tensor algebra of the finite $\rC^*$-correspondence $(\mathbb{C}^d, \mathbb{C})$ \cite[Example 2.7]{MS1}. In particular, the algebra $\A_d$ is residually finite-dimensional on account of \cite[Theorem 4.6]{CR}. If we also assume that $d\geq 2$, then $\A_d$ does not have an RFD $\rC^*$-cover dominated by $[\T_d, \id]$. Indeed, this can be read directly from the above short exact sequence.

A {\bf row contraction} is a $d$-tuple $T = (T_1, \ldots, T_d)$ of operators on a Hilbert space $\H$ with the property that \[\sum_{i=1}^d T_i T_i^*\leq I.\]When $T = (T_1, \ldots, T_d)$ consists of (co)isometries with pairwise orthogonal ranges, then $T$ is said to be a {\bf row (co)isometry}. If $T$ is a row isometry for which equality is achieved, then $T$ is said to be of {\bf Cuntz type}. Popescu proved that the algebra $\A_d$ is universal for unital operator algebras generated by a row contraction \cite[Theorem 2.1]{Popescu}. That is, for any row contraction $(T_1, \ldots, T_d)$, there is a unital representation $\rho:\A_d\rightarrow B(\H)$ such that $\rho(L_i) = T_i$ for each $i=1,\ldots, d$. Leveraging this universal property, we can extrapolate a finer conclusion than that of Theorem \ref{t:fin-tens-meet}.

%The main technical argument that we require is that the subset of RFD $\rC^*$-covers in $\cstarlattice(\A_d)$ is not closed under countable meets.

\begin{theorem}\label{t:pop-rfd-min}
Fix a positive integer $d\ge1$ and let $\A_d\subset B(\fF_d^2)$ denote the non-commutative disc algebra. Then the following statements hold.\begin{enumerate}
    \item [(1)] There is a sequence $[\R_m, \upsilon_m], m\geq1,$ of residually finite-dimensional C*-covers for $\A_d$ such that \[\bigwedge_{m\in\bbN}[\R_m, \upsilon_m] = [\T_d\oplus\T_d, \eta]\]where $\eta:\A_d\rightarrow B(\fF_d^2\oplus\fF_d^2)$ is such that $\eta(L_i) = L_i \oplus L_i^*$ for each $i=1,\ldots, d$.
    \item [(2)] If $d\geq 2$, then the space of residually finite-dimensional C*-covers for $\A_d$ has no minimal element.
\end{enumerate}
\end{theorem}

\begin{proof}
(1) For each $k\in\bbN$, let $\fL_k\subset\fF_d^2$ denote the closed linear span of $\{e_w : w\in\bbF_d^+, \ |w|\leq k\}$. For every $k\in\bbN, j=1,\ldots, d,$ define \[S_{j,k} = P_{\fL_k}L_j\mid_{\fL_k}\] and note that $(S_{1,k},\ldots, S_{d,k})$ is a row contraction. Furthermore, for each $n\in\bbN$, \[(T_{m,1}, \ldots, T_{m,d}) = \bigoplus_{k\geq m} (S_{1,k}, \ldots, S_{d,k}),\]is also a row contraction. A straightforward computation reveals that, for each $k\in\bbN$, the subspace $\fL_k$ is coinvariant for $\A_d$. Since $P_{\fL_k}$ converges to the identity operator in the strong operator topology, there is a completely isometric representation \[\upsilon_m: \A_d\rightarrow \ol{\text{alg}}\{ I, T_{m,1}, \ldots, T_{m,d}\}, \quad m\in\bbN,\]such that $\upsilon_m(L_i) = T_{m,i}$ for each $1\leq i\leq d$. Thus, letting $\R_m:=\rC^*(\upsilon_m(\A_d))$, we find that $\{[\R_m, \upsilon_m]: m\geq1\}$ is a collection of $\rC^*$-covers for $\A_d$. As $\{\fL_k : k\in\bbN\}$ is a family of finite-dimensional Hilbert spaces and \[\R_m\subset \prod_{k\geq m} B(\fL_k),\] it follows that the $\rC^*$-covers $\{ [\R_m, \upsilon_m] : m\geq 1\}$ are RFD.

The remainder of the proof is dedicated to achieving a concrete representation of \[\bigwedge_{m\in\bbN}[\R_m, \upsilon_m].\]To this end, for each $m\in\bbN$, note that there is a projection mapping \[\prod_{k\geq m} B(\fL_k)\rightarrow\prod_{k\geq{m+1}} B(\fL_k)\]that restricts to a morphism of $\rC^*$-covers. That is, we have $[\R_m, \upsilon_m]$ determines a downward directed sequence in $\cstarlattice(\A_d)$. By Proposition \ref{p:meet}, we have\[\bigwedge_{m\in\bbN}[\R_m, \upsilon_m] = \left[\varinjlim \R_m, \upsilon\right]\]for some completely isometric map $\upsilon:\A_d\rightarrow B(\H)$. We let $\R = \varinjlim \R_m$ and $(V_1, \ldots, V_d)$ denote the operator $d$-tuple $(\upsilon(L_1), \ldots, \upsilon(L_d))$. %$\pi_m:(\R_m, \upsilon_m)\rightarrow(\R, \upsilon)$ denote the corresponding morphism of $\rC^*$-covers.

By the universal property of $\A_d$, there is a completely isometric representation \[\eta:\A_d\rightarrow B(\fF_d^2\oplus\fF_d^2)\] such that $\eta(L_i) = L_i\oplus L_i^*$ for each $i=1,\ldots, d$. Note that the unital $\rC^*$-algebra generated by $\eta(\A_d)$ is $\T_d \oplus \T_d$ and thus, $[\T_d\oplus\T_d, \eta]$ defines a $\rC^*$-cover for $\A_d$. We claim that\[\bigwedge_{m\in\bbN}[\R_m, \upsilon_m] = [\T_d\oplus\T_d, \eta].\]To achieve this, we prove a Wold-like decomposition for the row contraction $(V_1, \ldots, V_d)$. First, one can directly verify that 
\[I- \sum_{i=1}^d S_{i,k}S_{i,k}^* \quad \text{and} \quad I- \sum_{i=1}^d S_{i,k}^*S_{i,k}\]
are non-zero projections for each $k\in\bbN$. Thus, in the direct limit, $\R = \varinjlim \R_m$, we have that \[ P = I - \sum_{i=1}^d V_i V_i^* \quad \text{and} \quad Q = I - \sum_{i=1}^d V_i^*V_i\]are non-zero projections as well. In addition, a tedious check reveals that 
\[S_{l,k}^*\left(I-\sum_{i=1}^d S_{i,k}S_{i,k}^*\right) = 0, \quad (S_{l, k}^*)^j S_{l,k}^j \left(I-\sum_{i=1}^d S_{i,k}S_{i,k}^*\right) = I-\sum_{i=1}^d S_{i,k}S_{i,k}^*\]
whenever $k\in\bbN, 1\leq l\leq d,$ and $1\leq j\leq k-1$. Thus, we must have that 
\[V_i^*P = 0 \quad \text{and} \quad (V_i^*)^j V_i^j P = P, \quad j\in\bbN, \ 1\leq i\leq d.\] As a consequence, we have 
\[\H_1:= \R P\H = \bigoplus_{w\in\bbF_d^+} V_w\M\]
where $\M = \bigcap_{i=1}^d \ker(V_i^*)$ and $V_w = V_{i_1}V_{i_2}\ldots V_{i_m}$ whenever $w=i_1i_2\ldots i_m$. Moreover, $\H_1$ is a wandering subspace for $(V_1, \ldots, V_d)$ and $(V_1\mid_{\H_1}, \ldots, V_d\mid_{\H_1})$ is a row isometry. A symmetric argument shows \[\H_2:= \R Q\H = \bigoplus_{w\in\bbF_d^+} V_w^*\N\]where $\N = \bigcap_{i=1}^d \ker(V_i)$, and that $(V_1\mid_{\H_2}, \ldots, V_d\mid_{\H_2})$ is a row coisometry.

We claim that $\H_1$ is orthogonal to $\H_2$. To this end, observe that, for each $j,k\in\bbN$ such that $1\leq j\leq k-2$, and each $1\leq l\leq d$,\[\left( I - \sum_{i=1}^d S_{i,k}^*S_{i,k}\right) S_{l,k}^j\left( I - \sum_{i=1}^d S_{i,k}S_{i,k}^*\right) = 0.\]In the direct limit $\R$, we must then have $QV_i^j P = 0$ for each $j\in\bbN$ and each $1\leq i\leq d.$ Consequently, \[\langle V_i^j P h, (V_i^*)^l Qk\rangle = \langle Q V_i^{j+l} P h, k\rangle = 0, \quad h,k\in\H,\]and $\H_1$ is orthogonal to $\H_2$ as desired.

This will now provide us with a direct sum decomposition for the $d$-tuple $V = (V_1, \ldots, V_d)$. When $h\in\bigcap_{i=1}^d\ker(V_i)$, we have that \[Qh = \left( I- \sum_{i=1}^d V_i^*V_i\right) h = h\]and so $\bigcap_{i=1}^d \ker V_i\subset \H_2.$ Similarly, it is straightforward to verify that \[\bigcap_{i=1}^d \ker (V_i^*)\subset \H_1.\]In particular, letting $\H_3 = (\H_1\oplus\H_2)^\perp$, we have that $V\mid_{\H_3}$ is of Cuntz type. Hence, we have a unitary equivalence \[V = V\mid_{\H_1} + V\mid_{\H_2} + V\mid_{\H_3} \simeq (L\otimes I_{P\H}) \oplus (L^*\otimes I_{Q\H}) \oplus U\]where $U$ is a $d$-tuple of Cuntz isometries.

By \cite[Theorem 3.1]{Popescu}, the quotient mapping $\T_d\rightarrow\O_d$ is completely isometric on $\A_d$. Thus, the simplicity of $\O_d$ implies that it must be the $\rC^*$-envelope for $\A_d$. In particular, there is a pair of $\ast$-homomorphisms \[\theta_1: \T_d\oplus\T_d\rightarrow\O_d \quad \text{and} \quad \theta_2:\O_d\rightarrow\rC^*(I, U)\]such that $\theta_1(L\oplus L^*)$ is a $d$-tuple of Cuntz isometries and $(\theta_2\theta_1)(L\oplus L^*) = U$. Therefore, we have a generator-preserving $*$-isomorphism \[\rC^*(\upsilon(\A_d)) \simeq \rC^*(I, L\oplus L^*\oplus U) = (id\oplus(\theta_2\theta_1))(\T_d\oplus\T_d)\simeq\T_d\oplus\T_d.\] In particular, this establishes an equivalence \[
\bigwedge_{n\in\bbN}[\R_n, \upsilon_n] = [\T_d\oplus\T_d, \eta]
\] of $\rC^*$-covers for $\A_d$, as desired.

(2) For this implication, assume that $d\geq 2$. Let $[\R, \gamma]$ be the meet of all RFD $\rC^*$-covers for $\A_d$. We claim that $\R$ is not an RFD $\rC^*$-algebra. By statement (1), we must have that \[[\R, \gamma]\preceq [\T_d\oplus\T_d, \eta].\] Thus, $\R$ must be a quotient of $\T_d\oplus\T_d$. Since $\T_d\oplus\T_d$ has no quotients that are RFD, we may conclude that $\R$ cannot be RFD.
\end{proof}

%We remark that we are uncertain whether $(\rC^*(I, L\oplus L^*), \iota)$ is equivalent to $(\fR_{min}, \varepsilon_r)$.

It is conceivable that our argument for $\A_d$ can be strengthened to much larger families of operator algebras. Consider a finite (directed) graph $G$ that has a cycle with an entry and consider the tensor algebra $\T_{X_G}^+$ of the graph correspondence $X_G$. The tensor algebra $\T_{X_G}^+$ admits a convenient representation as a so-called Toeplitz--Cuntz--Krieger family (see \cite[Section 3]{DS} for more details). As $G$ is finite and there is a cycle with an entry in $G$, the graph $\rC^*$-algebra $\rC^*(G)$, which is the $\rC^*$-envelope, fails to be RFD \cite{BelShul}. Furthermore, the Toeplitz algebra $\T_{X_G}$ cannot be RFD either, as it contains proper isometries. From here, it is conceivable that combining our proof strategy with the argument of \cite[Theorem 3.2]{DS} implies that the meet of all RFD $\rC^*$-covers for $\T_{X_G}^+$ is not RFD either.

All of the examples we have considered in this paper are tensor algebras of C*-correspondences. This should not suggest anything to the reader, as it is a broad class of frequently studied operator algebras. To rectify this, note that the adjoint $\A^*$ of any of our previous examples can share similar conclusions by \cite[Proposition 2.8]{HumRam}. For the adjoint $\A_d^*$, $d\geq 2$, of the non-commutative disc algebra, we have that $\A_d^*$ is not completely isometrically isomorphic to a tensor algebra as it is not semi-Dirichlet. One could also consider $\A = \A_d \oplus \A_d^*$, whose RFD $\rC^*$-covers do not form a lattice, nor does there exist a minimal RFD $\rC^*$-cover \cite[Theorem 2.10]{HumRam}. In this case, both $\A$ and $\A^*$ fail to be semi-Dirichlet, and so neither is a tensor algebra of a C*-correspondence. However, as we shall outline below, there are many other far-reaching classes of algebras where our arguments could conceivably be adapted. %We briefly record examples that are not tensor algebras to rectify this.

In accordance with this, as well as the developments of this paper, we present the two natural questions that remain.

\begin{question}
    Does there exist an RFD operator algebra with the property that not every $\rC^*$-cover is RFD, and yet the RFD $\rC^*$-covers form a (complete) lattice?
\end{question}

\begin{question}
    Is there an RFD operator algebra for which the meet $[R, \gamma]$ of all RFD $\rC^*$-covers is RFD, yet $[R, \gamma]$ is not the $\rC^*$-envelope?
\end{question}

A positive answer to either question would likely need to overcome the dilation-theoretic obstacle that we present in this paper. With regard to the latter question, one would require an operator algebra $\A$ admitting a sufficiently large family $\F$ of finite-dimensional representations where each $\pi\in\F$ cannot be non-trivially dilated to a finite-dimensional representation, while simultaneously allowing for the existence of a non-trivial (infinite-dimensional) dilation. Unfortunately, many of the most well-understood classes of non self-adjoint operator algebras fall outside this class (see \cite[Proposition 4.1]{CDO},\cite[Theorem 7.1]{DDSS},\cite[Section 3]{DS},\cite[Section 3]{MP} for example). Furthermore, for $\iota:=\bigoplus_{\pi\in\F}\pi$ to be completely isometric, it would be necessary that $\{\dim\pi: \pi\in\F\}$ be unbounded. Otherwise, the induced $\rC^*$-cover $\rC^*(\iota(\A))$ would be subhomogeneous and, in particular, this would force the quotient $\rC^*_e(\A)$ to be residually finite-dimensional. However, at the time of writing, the authors are unaware of how to construct an example of a single representation satisfying the dilation-theoretic constraints outlined above.

\section*{Acknowledgements} The authors would like to thank David Sherman for discussions and ideas that led to this project and Mitch Hamidi for many discussions and his support of this project. The first author was supported by the NSERC Discovery Grant 2024-03883. The second author was
supported by the NSERC Discovery Grant 2019-05430. The third author was supported by an NSERC Postdoctoral Fellowship.

\end{document}